\documentclass{amsart}
\usepackage{amsfonts, amsbsy, amsmath, amssymb}

\usepackage[figuresright]{rotating}

\usepackage{multirow}

\input prepictex.tex
\input pictex.tex
\input postpictex.tex

\newtheorem{thm}{Theorem}[section]
\newtheorem{lem}[thm]{Lemma}
\newtheorem{cor}[thm]{Corollary}

\numberwithin{equation}{section}

\theoremstyle{definition}

\begin{document}

\title{Finite Modules over $\Bbb Z[t,t^{-1}]$}

\author{Xiang-dong Hou}
\address{Department of Mathematics  and Statistics,
University of South Florida, Tampa, FL 33620}
\email{xhou@usf.edu}


 
\begin{abstract}
Let $\Lambda=\Bbb Z[t,t^{-1}]$ be the ring of Laurent polynomials over $\Bbb Z$. We classify all $\Lambda$-modules $M$ with $|M|=p^n$, where $p$ is a primes and $n\le 4$. Consequently, we have a classification of Alexander quandles of order $p^n$ for $n\le 4$. 
\end{abstract}

\maketitle

\section{Introduction}

Let $\Lambda=\Bbb Z[t,t^{-1}]$ be the ring of Laurent polynomials over $\Bbb Z$, which is also the group ring over $\Bbb Z$ of the infinite cyclic group. Each $\Lambda$-module is uniquely determined by a pair $(M,\alpha)$, where $M$ is an abelian group and $\alpha\in\text{Aut}_\Bbb Z(M)$. The resulting $\Lambda$-module, denoted by $M_\alpha$, is $M$ with a scalar multiplication defined by $tx=\alpha(x)$, $x\in M$.
If two $\Lambda$-modules $M_\alpha$ and $N_\beta$ are isomorphic, where $\alpha\in \text{Aut}_\Bbb Z(M)$ and $\beta\in \text{Aut}_\Bbb Z(N)$, then $M\cong N$ as abelian groups. Moreover, for $\alpha,\beta\in\text{Aut}_\Bbb Z(M)$, $M_\alpha\cong M_\beta$ if and only if $\alpha$ and $\beta$ are conjugate in 
$\text{Aut}_\Bbb Z(M)$. Thus, to classify $\Lambda$-modules with an underlying abelian group $M$ is to determine the conjugacy classes of $\text{Aut}_\Bbb Z(M)$. 

Our interest in finite $\Lambda$-modules comes from topology. In knot theory, a {\em quandle} is defined to be a set of $Q$ equipped with a binary operation $*$ such that for all $x,y,z\in Q$,
\begin{itemize}
  \item [(i)] $x*x=x$,
  \item [(ii)] $(\ )*y$ is a permutation of $Q$,
  \item [(iii)] $(x*y)*z=(x*z)*(y*z)$.
\end{itemize}  
Finite quandles are used to color knots; the number of colorings of a knot $K$ by a finite quandle $Q$ is an invariant of $K$ which allows us to distinguish inequivalent knots effectively \cite{CKS, CEHSY}. 

An {\em Alexander quandle} is a $\Lambda$-module $M$ with a quandle operation defined by $x*y=tx+(1-t)y$, $x,y\in M$. The following theorem is of fundamental importance.

\begin{thm}[\cite{Nel03}] \label{T1.1}
Two finite Alexander quandles $M$ and $N$ are isomorphic if and only if $|M|=|N|$ and the $\Lambda$-modules $(1-t)M$ and $(1-t)N$ are isomorphic.
\end{thm}

Therefore, the classification of finite Alexander quandles is essentially the classification of finite $\Lambda$-modules.

The classification of finite $\Lambda$-modules can be reduced to that of $\Lambda$-modules of order $p^n$, where $p$ is a prime; see section 2. The same is true for the classification of finite Alexander quandles; see section 4. Finite Alexander quandles have been classified for orders up to 15 in \cite{Nel03} and for order 16 in \cite{Mur-Nel08, Mur-Nel09}. Also known is the classification of connected Alexander quandles of order $p^2$ \cite{Gra04, Nel03}. (A finite Alexander quandle is called {\em connected} if $1-t\in\text{Aut}_\Bbb Z(M)$.) The purpose of the present paper is to classify all $\Lambda$-modules and Alexander quandles of order $p^n$, $n\le 4$. The details of the classification are given in Table 1 in the appendix. For a snapshot, there are $5p^4-2p^3-2p-1$ nonisomorphic $\Lambda$-modules of order $p^4$ and there are $5p^4-6p^3+p^2-6p-1$ nonisomorphic Alexander quandles of order $p^4$.  

In section 2, we show that every finite $\Lambda$-module has a unique decomposition where each direct summand $M$ has the following properties:
\begin{itemize}
  \item [(i)] $|M|=p^n$ for some prime $p$ and integer $n>0$.
  \item [(ii)] When treating $t$ as an element of $\text{End}_{\Bbb Z_p}(M/pM)$, the minimal polynomial of $t$ is a power of some irreducible $f\in\Bbb Z_p[X]$.
\end{itemize}  
In section 3, we classify such $\Lambda$-modules $M$ of order $p^n$ with $n\le 4$. In section 4, we derive the classification of Alexander quandles of order $p^n$, $n\le 4$, from the results of section 3. For this purpose, we prove the following fact which is of interest in its own right: Given a finite $\Lambda$-module $N$ and an integer $l>0$, a necessary and sufficient condition for the existence of a $\Lambda$-module $M\supset N$ such that $(1-t)M=N$ and $|M/N|=l$ is $|N/(1-t)N|\bigm | l$. 

In our notation, the letter $t$ is reserved for the element $t\in\Lambda=\Bbb Z[t,t^{-1}]$. The group of units of a ring $R$ is denoted by $R^\times$; the set of all $m\times n$ matrices over $R$ is denoted by ${\text M}_{m\times n}(R)$.


\section{Decomposition of Finite $\Lambda$-Modules}

Let $M$ be a finite $\Lambda$-module. For each prime $p$, let 
\[
M_p=\{x\in M: p^nx=0\ \text{for some}\ n\ge 0\}.
\]
It is quite obvious that
\begin{equation}\label{2.0}
M=\bigoplus_p M_p.
\end{equation}
Moreover, two finite $\Lambda$-modules $M$ and $N$ are isomorphic if and only if $M_p\cong N_p$ for all primes $p$.

\begin{thm}\label{T2.1}
Let $M$ be a finite $\Lambda$-module with $|M|=p^n$. For each irreducible $f\in\Bbb Z_p[X]$, let $\overline f\in\Bbb Z[X]$ be a lift of $f$ and define
\begin{equation}\label{2.1}
M_f=\{x\in M: \overline f(t)^mx=0\ \text{for some}\ m\ge 0\}.
\end{equation}
Then
\begin{equation}\label{2.2}
M=\bigoplus_fM_f,
\end{equation}
where $f$ runs over all irreducible polynomials in $\Bbb Z_p[X]$. Moreover, if $N$ is another finite $\Lambda$-module whose order is a power of $p$, then $M\cong N$ if and only if $M_f\cong N_f$ for all irreducible $f\in\Bbb Z_p[X]$.
\end{thm}

\noindent{\bf Note.} $M_f$ depends only on $f$ but not on $\overline f$. Also, $M_f=0$ unless $f$ divides the minimal polynomial of $t$ (viewed as an element of $\text{End}_{\Bbb Z_p}(M/pM)$).

\begin{proof}[Proof of Theorem~\ref{T2.1}]
$1^\circ$ Let the minimal polynomial of $t$ ($\in\text{End}_{\Bbb Z_p}(M/pM)$) be $f_1^{e_1}\cdots f_k^{e_k}$, where $f_1,\dots,f_k\in\Bbb Z_p[X]$ are distinct irreducibles and $e_1,\dots, e_k$ are positive integers. We claim that
\begin{equation}\label{2.3}
M=\bigoplus_{1\le i\le k}M_{f_i}.
\end{equation}

We first show that $\sum_{1\le i\le k}M_{f_i}$ is a direct sum. Assume that $x\in M_{f_i}\cap\bigl(\sum_{1\le j\le k,\, j\ne i}M_{f_j}\bigr)$.
Then there exists $m>0$ such that $\overline{f_i}(t)^m x=0$ and \break
$\bigl(\prod_{1\le j\le k,\, j\ne i}\overline{f_j}(t)\bigr)^mx=0$. Since $\text{gcd}\bigl( f_i,
\prod_{1\le j\le k,\, j\ne i}f_j\bigr)=1$, there exist $u,v\in\Bbb Z_p[X]$ such that
\[
uf_i^m+v\Bigl(\prod_{\substack{1\le j\le k\cr j\ne i}}f_j\Bigr)^m=1.
\]
Let $\overline u,\overline v\in\Bbb Z[X]$ be arbitrary lifts of $u,v$, respectively. Then 
\[
\overline u\overline{f_i}^m+\overline v\Bigl(\prod_{\substack{1\le j\le k\cr j\ne i}}\overline{f_j}\Bigr)^m\equiv 1\pmod p.
\]
Therefore
\[
\overline u(t)\overline{f_i}(t)^m+\overline v(t)\Bigl(\prod_{\substack{1\le j\le k\cr j\ne i}}\overline{f_j}(t)\Bigr)^m\in\text{Aut}_{\Bbb Z}(M).
\]
Since
\[
\Bigl[\overline u(t)\overline{f_i}(t)^m+\overline v(t)\Bigl(\prod_{\substack{1\le j\le k\cr j\ne i}}\overline{f_j}(t)\Bigr)^m\Bigr]x=0,
\]
we have $x=0$.

Now we prove that $M=\sum_{1\le i\le k}M_{f_i}$. There exist $u_1,\dots,u_k\in\Bbb Z_p[X]$ such that
\[
\sum_{1\le i\le k}u_i\Bigl(\prod_{\substack{1\le j\le k\cr j\ne i}}f_j^{e_j}\Bigr)^n=1.
\]
Let $\overline{u_i}\in\Bbb Z[X]$ be a lift of $u_i$ and let $F_i=\prod_{1\le j\le k,\, j\ne i}\overline{f_j}^{e_j}$. Then
\[
\sum_{1\le i\le k}\overline{u_i}F_i^n\equiv 1\pmod p.
\]
Thus $\sum_{1\le i\le k}\overline{u_i}(t)F_i(t)^n\in\text{Aut}_\Bbb Z(M)$. It follows that 
\begin{equation}\label{2.4}
M=\sum_{1\le i\le k}F_i(t)^nM.
\end{equation}
Since $\bigl(\prod_{1\le j\le k}\overline{f_j}(t)^{e_j}\bigr)M\subset pM$, we have
\[
\overline{f_i}(t)^{e_in}F_i(t)^nM=\Bigl(\prod_{1\le j\le k}\overline{f_j}(t)^{e_j}\Bigr)^nM\subset p^nM=0.
\]
Thus $F_i(t)^nM\subset M_{f_i}$. Then it follows from \eqref{2.4} that $M=\sum_{1\le i\le k}M_{f_i}$.

\medskip

$2^\circ$ Let $N$ be another finite $\Lambda$-module whose order is a power of $p$. If there is a $\Lambda$-module isomorphism $\phi:M\to N$, then for each irreducible $f\in\Bbb Z_p[X]$, $\phi|_{M_f}:M_f\to N_f$ is an isomorphism. Conversely, if $M_f\cong N_f$ for all irreducible $f\in\Bbb Z_p[X]$, then by \eqref{2.2}, $M\cong N$.  
\end{proof}  


\section{Classification of $\Lambda$-Modules of Order $p^n$, $n\le 4$}

\subsection{The automorphism group of a finite abelian group}\

Let $p$ be a prime and let $m\ge n>0$ be integers. Elements of $\Bbb Z_{p^m}$ can be viewed as elements of $\Bbb Z_{p^n}$ via the homomorphism
\[
\begin{array}{cccc}
\Bbb Z_{p^m}&\longrightarrow&\Bbb Z_{p^n}\cr
a+p^m\Bbb Z&\longmapsto &a+p^n\Bbb Z,&\ a\in\Bbb Z.
\end{array}
\]
Likewise, elements of $p^{m-n}\Bbb Z_{p^n}$ can be viewed as elements of $\Bbb Z_{p^m}$ via the embedding
\[
\begin{array}{cccc}
p^{m-n}\Bbb Z_{p^n}&\longrightarrow&\Bbb Z_{p^m}\cr
p^{m-n}(a+p^n\Bbb Z)&\longmapsto &p^{m-n}a+p^m\Bbb Z,&\ a\in\Bbb Z.
\end{array}
\]
We shall adopt these conventions hereafter.

Let $M=\Bbb Z_{p^{e_1}}^{n_1}\times\cdots\times\Bbb Z_{p^{e_k}}^{n_k}$, where $n_i>0$ and $e_1>\cdots>e_k>0$. Elements of $\text{End}_{\Bbb Z}(M)$ are of the form
\[
\begin{array}{ccccc}
\sigma_A:& M &\longrightarrow &M\cr
& \left[\begin{matrix} x_1\cr\vdots \cr x_k\end{matrix}\right]&\longmapsto& A\left[\begin{matrix} x_1\cr\vdots \cr x_k\end{matrix}\right],&\ x_i\in\Bbb Z_{p^{e_i}}^{n_i},
\end{array}
\]
where
\begin{equation}\label{3.1}
A=\left[
\begin{matrix}
A_{11}&p^{e_1-e_2}A_{12}&\cdots&p^{e_1-e_k}A_{1k}\cr
A_{21}&A_{22}&\cdots&p^{e_2-e_1}A_{2k}\cr
\vdots&\vdots&&\vdots\cr
A_{k1}&A_{k2}&\cdots&A_{kk}
\end{matrix}\right],
\end{equation}
and $A_{ij}\in\text{M}_{n_i\times n_j}(\Bbb Z_{p^{e_i}})$.
Let $\frak M(M)$ denote the set of all matrices of the form \eqref{3.1}. Then
\[
\begin{array}{ccc}
\frak M(M)&\longrightarrow&\text{End}_{\Bbb Z}(M)\cr
A&\longmapsto&\sigma_A
\end{array}
\]
is a ring isomorphism. Let $\text{GL}(M)$ denote the group of units of $\frak M(M)$. Of course $\text{GL}(M)\cong \text{Aut}_{\Bbb Z}(M)$ under the above isomorphism. It is known
\cite{Hil-Rhe07, Ran1907} (and also easy to prove) that
\[
\text{GL}(M)=\{A: A\ \text{is of the form \eqref{3.1} with}\ A_{ii}\in\text{GL}(n_i,\Bbb Z_{p^{e_i}}),\ 1\le i\le k\}.
\]
The modulo $p$ reduction from $\text{GL}(M)$ to $\text{GL}(n_1+\cdots+n_k,\Bbb Z_p)$ is denoted by $\overline{(\ )}$.
For each (monic) irreducible $f\in\Bbb Z_p[X]$ with $f\ne X$, define
\[
\begin{split}
\text{GL}(M)_f=\{A\in\text{GL}(M):\;& \text{the minimal polynomial of}\cr
& \overline A\in \text{GL}(n_1+\cdots+n_k,\Bbb Z_p)\ \text{is a power of}\ f\}.
\end{split}
\]
If $\lambda^{(i)}=(\lambda_{i1},\lambda_{i2},\dots)$ is a partition of the integer $n_i/\deg f$, $1\le i\le k$, we define
\[
\begin{split}
\text{GL}(M)_f^{\lambda^{(1)}\dots\lambda^{(k)}}=\{\text{$A$ as in \eqref{3.1}} :\; &\text{the elementary divisors of $\overline {A_{ii}}$}\cr
&\text{are $f^{\lambda_{i1}}, 
f^{\lambda_{i2}},\dots$},\ 1\le i\le k\}.
\end{split}
\]  

In this setting, our goal is to determine the $\text{GL}(M)$-conjugacy classes in $\text{GL}(M)_f$. We will proceed according to the structure of $(M,+)$.

\subsection{$(M,+)=\Bbb Z_{p^e}$}\

In this case we must have $f=X-a$, $a\in\Bbb Z_p^\times$. The conjugacy classes in $\text{GL}(M)_f$ are represented by
\[
[b],\qquad b\in\Bbb Z_{p^e},\ b\equiv a \pmod p.
\]

\subsection{$(M,+)=\Bbb Z_p^n$}\

In this case we must have $\deg f\mid n$. The conjugacy classes in $\text{GL}(M)_f$ are represented by the rational canonical forms in $\text{GL}(n,\Bbb Z_p)$ with elementary divisors $f^{\lambda_1}, f^{\lambda_2},\dots$, where $\lambda_1\ge \lambda_2\ge\cdots>0$ is a partition of $n/\deg f$.

\subsection{$(M,+)=\Bbb Z_{p^e}\times \Bbb Z_p$, $e>1$}\

In this case, $\deg f=1$.

\begin{thm}\label{T3.1}
Assume $(M,+)=\Bbb Z_{p^e}\times\Bbb Z_p$, $e>1$, and $f=X-a$, $a\in\Bbb Z_p^\times$. The conjugacy classes in $\text{\rm GL}(M)_f$ are represented by the following matrices:
\begin{itemize}
  \item [(i)] $\left[\begin{matrix} b\cr &b\end{matrix}\right]+\left[\begin{matrix} p^{e-1}\alpha&0\cr 0&0\end{matrix}\right]$,\kern 5mm  $0<b<p^{e-1}$, $b\equiv a\pmod p$, $\alpha\in\Bbb Z_p$.
  \medskip  

  \item [(ii)] $\left[\begin{matrix} b\cr &b\end{matrix}\right]+\left[\begin{matrix} 0&0\cr 1&0\end{matrix}\right]$,\kern 5mm  $0<b<p^{e-1}$, $b\equiv a\pmod p$.
  \medskip

  \item  [(iii)] $\left[\begin{matrix} b\cr &b\end{matrix}\right]+\left[\begin{matrix} 0& p^{e-1}\cr \gamma&0\end{matrix}\right]$,\kern 5mm  $0<b<p^{e-1}$, $b\equiv a\pmod p$, $\gamma\in\Bbb Z_p$.
\end{itemize}  
\end{thm}

\begin{proof} Elements of $\text{GL}(M)_f$ are of the form
\[
A(b,\alpha,\beta,\gamma):=\left[\begin{matrix} b&0\cr 0&b\end{matrix}\right]+\left[\begin{matrix} p^{e-1}\alpha&p^{e-1}\beta\cr \gamma&0\end{matrix}\right],
\]
where $0<b<p^{e-1}$, $b\equiv a\pmod p$, $\alpha,\beta,\gamma\in\Bbb Z_p$. Let $A(b,\alpha,\beta,\gamma), A(b,\alpha',\beta',\gamma')\in\text{GL}(M)_f$ and
\[
P=\left[
\begin{matrix} x&p^{e-1}y\cr
z&w\end{matrix}\right]\in\text{GL}(M),\quad x\in\Bbb Z_{p^e}^\times,\ w\in\Bbb Z_p^\times,\ y,z\in\Bbb Z_p.
\]
The equation $PA(b,\alpha,\beta,\gamma)=A(b,\alpha',\beta',\gamma')P$ is equivalent to 
\[
\left[
\begin{matrix}
p^{e-1}(x\alpha+y\gamma)& p^{e-1}x\beta\cr
w\gamma&0\end{matrix}\right]=
\left[
\begin{matrix}
p^{e-1}(\alpha'x+\beta'z)& p^{e-1}\beta'w\cr
\gamma'x&0\end{matrix}\right].
\]
The above equation can be written as a matrix equation over $\Bbb Z_p$:
\[
\left[
\begin{matrix}
x\alpha+y\gamma& x\beta\cr
w\gamma&0\end{matrix}\right]=
\left[
\begin{matrix}
\alpha'x+\beta'z& \beta'w\cr
\gamma'x&0\end{matrix}\right],
\]
equivalently,
\[
\left[
\begin{matrix} x&y\cr 0&w\end{matrix}\right]
\left[
\begin{matrix} \beta&\alpha\cr 0&\gamma\end{matrix}\right]
=\left[
\begin{matrix} \beta'&\alpha'\cr 0&\gamma'\end{matrix}\right]
\left[
\begin{matrix} w&z\cr 0&x\end{matrix}\right].
\]
So, $A(b,\alpha,\beta,\gamma)$ and $A(b,\alpha',\beta',\gamma')$ are conjugate if and only if there exist $x,y\in\Bbb Z_p^\times$ and $y,z\in\Bbb Z_p$ such that
\begin{equation}\label{3.2}
\left[
\begin{matrix} x&y\cr 0&w\end{matrix}\right]
\left[
\begin{matrix} \beta&\alpha\cr 0&\gamma\end{matrix}\right]
\left[
\begin{matrix} w^{-1}&z\cr 0&x^{-1}\end{matrix}\right]=
\left[
\begin{matrix} \beta'&\alpha'\cr 0&\gamma'\end{matrix}\right].
\end{equation}
Let $\mathcal M=\{\left[\begin{smallmatrix}\beta&\alpha\cr 0&\gamma\end{smallmatrix}\right]:\alpha,\beta,\gamma\in\Bbb Z_p\}$. For $A=\left[\begin{smallmatrix}\beta&\alpha\cr 0&\gamma\end{smallmatrix}\right],\ A'=\left[\begin{smallmatrix}\beta'&\alpha'\cr 0&\gamma'\end{smallmatrix}\right]\in\mathcal M$, say $A\sim A'$ if \eqref{3.2} is satisfied for some $x,w\in\Bbb Z_p^\times$ and $y,z\in\Bbb Z_p$. It is easy to see that the $\sim$~equivalence classes in $\mathcal M$ are represented by
\begin{itemize}
  \item [(i)] $\left[\begin{matrix}0&\alpha\cr 0&0\end{matrix}\right],\quad\alpha\in\Bbb Z_p$,
  \medskip
  \item [(ii)]  $\left[\begin{matrix}0\cr &1\end{matrix}\right]$,
  \medskip
  \item [(iii)] $\left[\begin{matrix}1\cr &\gamma\end{matrix}\right],\quad \gamma\in\Bbb Z_p$.
\end{itemize}  
These matrices correspond to the representatives of the conjugacy classes in $\text{GL}(M)_f$ stated in the theorem.
\end{proof}

\subsection{$(M,+)=\Bbb Z_{p^2}^2$}\

In this case, $\deg f=1$ or $2$.

\begin{thm}\label{T3.2}
Assume $(M,+)=\Bbb Z_{p^2}^2$.
\begin{itemize}
  \item [(i)] Let $f=X-a$, $a\in\Bbb Z_p^\times$. Then the conjugacy classes in $\text{\rm GL}(M)_f^{(1,1)}$ are represented by the following matrices:
  \begin{itemize}  
    \item [(i.1)] $\left[\begin{matrix} b\cr &b\end{matrix}\right]+p\left[\begin{matrix} \alpha\cr &\delta\end{matrix}\right],\quad 0<b<p,\ b\equiv a\pmod p,\ 0\le \alpha\le \delta<p$.
    \medskip
    \item [(i.2)] $\left[\begin{matrix} b\cr &b\end{matrix}\right]+p\left[\begin{matrix} \alpha&1\cr &\alpha\end{matrix}\right],\quad 0<b<p,\ b\equiv a\pmod p,\ \alpha\in\Bbb Z_p$. 
    \medskip
    \item [(i.3)] $\left[\begin{matrix} b\cr &b\end{matrix}\right]+p\left[\begin{matrix} 0&1\cr -b_0&-b_1\end{matrix}\right],\quad 0<b<p,\ b\equiv a\pmod p,\ X^2+b_1X+b_0\in\Bbb Z_p[X]$ irreducible. 
  \end{itemize}    
 \item [(ii)] Let $f=X-a$, $a\in\Bbb Z_p^\times$. Then the conjugacy class in $\text{\rm GL}(M)_f^{(2)}$ are represented by
\[
\left[\begin{matrix} b&1\cr 0&b\end{matrix}\right]+p\left[\begin{matrix} \alpha&0\cr \gamma&0\end{matrix}\right],\quad 0<b<p,\ b\equiv a\pmod p,\ \alpha,\gamma\in\Bbb Z_p.
\]

\item [(iii)] Let $f=X^2+a_1X+a_0\in\Bbb Z_p[X]$ be irreducible. Then the conjugacy classes in $\text{\rm GL}(M)_f$ are represented by  
\[
\begin{split}
\left[\begin{matrix} 0&1\cr -b_0&-b_1\end{matrix}\right]+p\left[\begin{matrix} \alpha&\beta\cr 0&0\end{matrix}\right],\quad &0\le b_0,b_1<p,\ b_0\equiv a_0\pmod p,\cr
& b_1\equiv a_1\pmod p,\ \alpha,\beta\in\Bbb Z_p.
\end{split}
\]
\end{itemize}
\end{thm}

\begin{proof}
We remind the reader that in the proof, our notation is local in each of the three cases.

(i) Elements of $\text{GL}(M)_f^{(1,1)}$ are of the form
\[
A(b,\alpha,\beta,\gamma,\delta):=\left[\begin{matrix} b\cr &b\end{matrix}\right]+p\left[\begin{matrix} \alpha&\beta\cr \gamma&\delta\end{matrix}\right],\quad 0<b<p,\ b\equiv a\pmod p,\ \alpha,\beta,\gamma,\delta\in\Bbb Z_p.
\]
Clearly, $A(b,\alpha,\beta,\gamma,\delta)$ and $A(b,\alpha',\beta',\gamma',\delta')$ are conjugate in $\text{GL}(M)$ if and only if 
$\left[\begin{smallmatrix}\alpha&\beta\cr \gamma&\delta\end{smallmatrix}\right]$ and $\left[\begin{smallmatrix}\alpha'&\beta'\cr \gamma'&\delta'\end{smallmatrix}\right]$ are conjugate in
$\text{M}_{2\times 2}(\Bbb Z_p)$. The conjugacy classes in $\text{M}_{2\times 2}(\Bbb Z_p)$ are represented by 
\begin{itemize}
  \item [(1)] $\left[\begin{matrix} \alpha\cr &\delta\end{matrix}\right],\quad 0\le \alpha\le \delta<p$,
  \medskip
  \item [(2)] $\left[\begin{matrix} \alpha&1\cr &\alpha\end{matrix}\right],\quad \alpha\in\Bbb Z_p$,
  \medskip
  \item [(3)] $\left[\begin{matrix} 0&1\cr -b_0&-b_1\end{matrix}\right],\quad X^2+b_1X+b_0\in\Bbb Z_p[X]$ irreducible.
\end{itemize}  
These correspond to the matrices in (i.1) -- (i.3).

(ii) Elements of $\text{GL}(M)_f^{(2)}$ are conjugate to matrices of the form
\[
A(b,\alpha,\beta,\gamma,\delta):=\left[\begin{matrix}b&1\cr0&b\end{matrix}\right]+p\left[\begin{matrix}\alpha&\beta\cr \gamma&\delta\end{matrix}\right],\quad
0<b<p,\ b\equiv a\pmod p,\ \alpha,\beta,\gamma,\delta\in\Bbb Z_p.
\]
Assume that $P\in\text{GL}(M)=\text{GL}(2,\Bbb Z_{p^2})$ such that
\begin{equation}\label{3.3}
PA(b,\alpha,\beta,\gamma,\delta)= A(b,\alpha',\beta',\gamma',\delta')P.
\end{equation}
Then over $\Bbb Z_p$,
\[
\overline P\left[\begin{matrix}b&1\cr0&b\end{matrix}\right]=\left[\begin{matrix}b&1\cr0&b\end{matrix}\right]\overline P,
\]
which implies that $\overline P=\left[\begin{smallmatrix}c&d\cr0&c\end{smallmatrix}\right]$, $c\in\Bbb Z_p^\times$, $d\in\Bbb Z_p$. So
\[
P=\left[\begin{matrix}c&d\cr0&c\end{matrix}\right]+p\left[\begin{matrix}x&y\cr z&w\end{matrix}\right],\quad x,y,z,w\in\Bbb Z_p.
\]
Now \eqref{3.3} becomes
\[
p\left[\begin{matrix}c&d\cr0&c\end{matrix}\right]\left[\begin{matrix}\alpha&\beta\cr \gamma&\delta\end{matrix}\right]+p\left[\begin{matrix}x&y\cr z&w\end{matrix}\right]
\left[\begin{matrix}b&1\cr0&b\end{matrix}\right]=p\left[\begin{matrix}\alpha'&\beta'\cr \gamma'&\delta'\end{matrix}\right]\left[\begin{matrix}c&d\cr0&c\end{matrix}\right]+p
\left[\begin{matrix}b&1\cr0&b\end{matrix}\right]\left[\begin{matrix}x&y\cr z&w\end{matrix}\right].
\]
Over $\Bbb Z_p$, this becomes
\[
\left[\begin{matrix}c&d\cr0&c\end{matrix}\right]\left[\begin{matrix}\alpha&\beta\cr \gamma&\delta\end{matrix}\right]+\left[\begin{matrix}x&y\cr z&w\end{matrix}\right]
\left[\begin{matrix}b&1\cr0&b\end{matrix}\right]=\left[\begin{matrix}\alpha'&\beta'\cr \gamma'&\delta'\end{matrix}\right]\left[\begin{matrix}c&d\cr0&c\end{matrix}\right]+
\left[\begin{matrix}b&1\cr0&b\end{matrix}\right]\left[\begin{matrix}x&y\cr z&w\end{matrix}\right],
\]
which can be simplified as
\[
\left[\begin{matrix}c&d\cr0&c\end{matrix}\right]\left[\begin{matrix}\alpha&\beta\cr \gamma&\delta\end{matrix}\right]=\left[\begin{matrix}\alpha'&\beta'\cr \gamma'&\delta'\end{matrix}\right]\left[\begin{matrix}c&d\cr0&c\end{matrix}\right]+\left[\begin{matrix}z&w-x\cr0&-z\end{matrix}\right].
\]
Thus, $A(b,\alpha,\beta,\gamma,\delta)$ and $A(b,\alpha',\beta',\gamma',\delta')$ are conjugate if and only if there exist $d,z,w\in\Bbb Z_p$ such that
\begin{equation}\label{3.4}
\left[\begin{matrix}1&d\cr 0&1\end{matrix}\right]\left[\begin{matrix}\alpha&\beta\cr \gamma&\delta\end{matrix}\right]\left[\begin{matrix}1&-d\cr 0&1\end{matrix}\right]
+\left[\begin{matrix}z&w\cr 0&-z\end{matrix}\right]=\left[\begin{matrix}\alpha'&\beta'\cr \gamma'&\delta'\end{matrix}\right].
\end{equation}
For $A=\left[\begin{smallmatrix}\alpha&\beta\cr \gamma&\delta\end{smallmatrix}\right]$, $A'=\left[\begin{smallmatrix}\alpha'&\beta'\cr \gamma'&\delta'\end{smallmatrix}\right]
\in\text{M}_{2\times 2}(\Bbb Z_p)$, say $A\sim A'$ if \eqref{3.4} is satisfied for some $d,z,w\in\Bbb Z_p$. It is easy to see that the $\sim$ equivalence classes in 
$\text{M}_{2\times 2}(\Bbb Z_p)$ are represented by
\[
\left[\begin{matrix}\alpha&0\cr \gamma&0\end{matrix}\right],\quad \alpha,\beta\in\Bbb Z_p,
\]
which correspond to the matrices in (ii).

(iii) Elements of $\text{GL}(M)_f$ are conjugate to matrices of the form
\[
\begin{split}
A(f,\alpha,\beta,\gamma,\delta):=\left[\begin{matrix}0&1\cr -b_0&-b_1\end{matrix}\right]+p\left[\begin{matrix}\alpha&\beta\cr \gamma&\delta\end{matrix}\right],
\quad & 0\le b_0,b_1<p,\ b_0\equiv a_0\pmod p,\cr
&b_1\equiv a_1\pmod p,\ \alpha,\beta,\gamma,\delta\in\Bbb Z_p.
\end{split}
\]
Assume that $P\in\text{GL}(2,\Bbb Z_{p^2})$ such that
\begin{equation}\label{3.5}
PA(f,\alpha,\beta,\gamma,\delta)=A(f,\alpha',\beta',\gamma',\delta')P.
\end{equation}
Then over $\Bbb Z_p$, 
\[
\overline P \left[\begin{matrix}0&1\cr -b_0&-b_1\end{matrix}\right]=\left[\begin{matrix}0&1\cr -b_0&-b_1\end{matrix}\right]\overline P,
\]
which implies that
\[
P=uI+v\left[\begin{matrix}0&1\cr -b_0&-b_1\end{matrix}\right]+p\left[\begin{matrix}x&y\cr z&w\end{matrix}\right],\quad 0\le u,v<p,\ (u,v)\ne(0,0),\ x,y,z,w\in\Bbb Z_p.
\]
Now \eqref{3.5} becomes the following equation over $\Bbb Z_p$:
\begin{equation}\label{3.6}
\begin{split}
\Bigl(uI+v\left[\begin{matrix}0&1\cr -b_0&-b_1\end{matrix}\right]\Bigr)\left[\begin{matrix}\alpha&\beta\cr \gamma&\delta\end{matrix}\right]=\;&
\left[\begin{matrix}\alpha'&\beta'\cr \gamma'&\delta'\end{matrix}\right]\Bigl(uI+v\left[\begin{matrix}0&1\cr -b_0&-b_1\end{matrix}\right]\Bigr)\cr
&+\left[\begin{matrix}0&1\cr -b_0&-b_1\end{matrix}\right]\left[\begin{matrix}x&y\cr z&w\end{matrix}\right]-\left[\begin{matrix}x&y\cr z&w\end{matrix}\right]
\left[\begin{matrix}0&1\cr -b_0&-b_1\end{matrix}\right].
\end{split}
\end{equation}
The space $\bigl\{\left[\begin{smallmatrix}0&1\cr -b_0&-b_1\end{smallmatrix}\right]X-X\left[\begin{smallmatrix}0&1\cr -b_0&-b_1\end{smallmatrix}\right]:X\in\text{M}_{2\times 2}(\Bbb Z_p)\bigr\}$ has dimension 2 over $\Bbb Z_p$ \cite[\S4.4]{Hor-Joh91} and has a basis
\begin{gather*}
\left[\begin{matrix}0&1\cr -b_0&-b_1\end{matrix}\right]\left[\begin{matrix}1&0\cr 0&0\end{matrix}\right]-\left[\begin{matrix}1&0\cr 0&0\end{matrix}\right]
\left[\begin{matrix}0&1\cr -b_0&-b_1\end{matrix}\right]=\left[\begin{matrix}0&-1\cr -b_0&0\end{matrix}\right], \\
\left[\begin{matrix}0&1\cr -b_0&-b_1\end{matrix}\right]\left[\begin{matrix}0&1\cr 0&0\end{matrix}\right]-\left[\begin{matrix}0&1\cr 0&0\end{matrix}\right]
\left[\begin{matrix}0&1\cr -b_0&-b_1\end{matrix}\right]=\left[\begin{matrix}b_0&b_1\cr 0&-b_0\end{matrix}\right].
\end{gather*}
So \eqref{3.6} can be written as
\begin{equation}\label{3.7}
\begin{split}
\Bigl(uI+v\left[\begin{matrix}0&1\cr -b_0&-b_1\end{matrix}\right]\Bigr)\left[\begin{matrix}\alpha&\beta\cr \gamma&\delta\end{matrix}\right]=\;&
\left[\begin{matrix}\alpha'&\beta'\cr \gamma'&\delta'\end{matrix}\right] \Bigl(uI+v\left[\begin{matrix}0&1\cr -b_0&-b_1\end{matrix}\right]\Bigr)\cr
&+\left[\begin{matrix}-d&\frac c{b_0}-d\frac{b_1}{b_0}\cr c&d\end{matrix}\right],\quad c,d\in\Bbb Z_p.
\end{split}
\end{equation}
Thus $A(f,\alpha,\beta,\gamma,\delta)$ and $A(f,\alpha',\beta',\gamma',\delta')$ are conjugate if and only if there exist $0\le u,v\le v<p$, $(u,v)\ne(0,0)$, and $c,d\in\Bbb Z_p$ such that
\begin{equation}\label{3.8}
\Bigl(uI+v\left[\begin{matrix}0& \kern-1mm 1\cr -b_0&\kern-2mm -b_1\end{matrix}\right]\Bigr)\Bigl(\left[\begin{matrix}\alpha&\beta\cr \gamma&\delta\end{matrix}\right]-
\left[\begin{matrix}-d&\frac c{b_0}-d\frac{b_1}{b_0}\cr c&d\end{matrix}\right]\Bigr)\Bigl(uI+v\left[\begin{matrix}0&\kern-1mm 1\cr -b_0&\kern-2mm -b_1\end{matrix}\right]\Bigr)^{-1}=
\left[\begin{matrix}\alpha'&\beta'\cr \gamma'&\delta'\end{matrix}\right].
\end{equation}

For $A=\left[\begin{smallmatrix}\alpha&\beta\cr \gamma&\delta\end{smallmatrix}\right],\ A'=\left[\begin{smallmatrix}\alpha'&\beta'\cr \gamma'&\delta'\end{smallmatrix}\right]
\in\text{M}_{2\times 2}(\Bbb Z_p)$, say $A\sim A'$ if \eqref{3.8} is satisfied for some $0\le u,v<p$, $(u,v)\ne (0,0)$, and $c,d\in\Bbb Z_p$. It remains to show that the $\sim$ equivalence classes in $\text{M}_{2\times 2}(\Bbb Z_p)$ are represented by 
\[
\left[\begin{matrix}\alpha&\beta\cr 0&0\end{matrix}\right],\quad \alpha,\beta\in\Bbb Z_p.
\]

First, choose $u=1$, $v=0$, $c=-\gamma$, $d=-\delta$. Then the left side of \eqref{3.8} becomes $\left[\begin{smallmatrix}\alpha'&\beta'\cr 0&0\end{smallmatrix}\right]$ for some $\alpha',\beta'\in\Bbb Z_p$. 

Next, assume that \eqref{3.8} holds with $\gamma=\delta=\gamma'=\delta'=0$. We want to show that $(\alpha,\beta)=(\alpha',\beta')$. Taking the traces of the two sides of \eqref{3.8}, we have $\alpha=\alpha'$. Now \eqref{3.7} with $\alpha=\alpha'$ and $\gamma=\delta=\gamma'=\delta'=0$ gives 
\[
\left[\begin{matrix}u\alpha&u\beta\cr -vb_0\alpha&-vb_0\beta\end{matrix}\right] =\left[\begin{matrix}\alpha u-\beta'vb_0&\alpha v-\beta'(u-vb_1)\cr 0&0\end{matrix}\right]
+\left[\begin{matrix}-d&\frac c{b_0}-d\frac{b_1}{b_0}\cr c&d\end{matrix}\right].
\]
Taking the traces, we have $vb_0\beta=vb_0\beta'$. If $v\ne 0$, then $\beta=\beta'$. If $v=0$, then $c=d=0$, which implies $u\beta=u\beta'$. Since $u\ne 0$, we also have $\beta=\beta'$.
\end{proof}

\subsection{$(M,+)=\Bbb Z_{p^2}\times\Bbb Z_p^2$}\

In this case $\deg f=1$.

\begin{thm}\label{T3.3}
Assume $(M,+)=\Bbb Z_{p^2}\times \Bbb Z_p^2$ and $f=X-a$, $a\in\Bbb Z_p^\times$. 
\begin{itemize}
  \item [(i)] The conjugacy classes in $\text{\rm GL}(M)_f^{(1)(1,1)}$ are represented by the following matrices:
  \medskip
  \begin{itemize}  
    \item [(i.1)] $\left[\begin{matrix} b\cr &b\cr &&b\end{matrix}\right]+ \left[\begin{matrix} p\alpha&0&0\cr 0&0&0\cr 0&0&0\end{matrix}\right],\quad 0<b<p,\ b\equiv a\pmod p,\ \alpha\in\Bbb Z_p$.
    \medskip
    \item [(i.2)] $\left[\begin{matrix} b\cr &b\cr &&b\end{matrix}\right]+ \left[\begin{matrix} 0&0&0\cr 1&0&0\cr 0&0&0\end{matrix}\right],\quad 0<b<p,\ b\equiv a\pmod p$.
    \medskip
    \item [(i.3)] $\left[\begin{matrix} b\cr &b\cr &&b\end{matrix}\right]+ \left[\begin{matrix} 0&p&0\cr 0&0&0\cr \eta&0&0\end{matrix}\right],\quad 0<b<p,\ b\equiv a\pmod p,\ \eta\in\Bbb Z_p$.
    \medskip
    \item [(i.4)] $\left[\begin{matrix} b\cr &b\cr &&b\end{matrix}\right]+ \left[\begin{matrix} 0&p&0\cr 1&0&0\cr 0&0&0\end{matrix}\right],\quad 0<b<p,\ b\equiv a\pmod p$.
\end{itemize}
\item [(ii)] The conjugacy classes in $\text{\rm GL}(M)_f^{(1)(2)}$ are represented by the following matrices:
\medskip
  \begin{itemize}    
  \item [(ii.1)] $\left[\begin{matrix} b\cr &b&1\cr &&b\end{matrix}\right]+ \left[\begin{matrix} p\alpha&0&0\cr 0&0&0\cr 0&0&0\end{matrix}\right],\quad 0<b<p,\ b\equiv a\pmod p,\ \alpha\in\Bbb Z_p$. 
  \medskip
  \item [(ii.2)] $\left[\begin{matrix} b\cr &b&1\cr &&b\end{matrix}\right]+ \left[\begin{matrix} 0&0&0\cr 0&0&0\cr 1&0&0\end{matrix}\right],\quad 0<b<p,\ b\equiv a\pmod p$.
  \medskip
  \item [(ii.3)] $\left[\begin{matrix} b\cr &b&1\cr &&b\end{matrix}\right]+ \left[\begin{matrix} 0&p&0\cr 0&0&0\cr \eta&0&0\end{matrix}\right],\quad 0<b<p,\ b\equiv a\pmod p,\ \eta\in\Bbb Z_p$.
\end{itemize}
\end{itemize}
\end{thm}

\begin{proof}
(i) Elements of $\text{GL}(M)_f^{(1)(1,1)}$ are of the form
\[
A(b,\alpha,\dots,\eta):= \left[\begin{matrix} b\cr &b\cr &&b\end{matrix}\right]+ \left[\begin{matrix} p\alpha&p\beta&p\gamma\cr \delta&0&0\cr \eta&0&0\end{matrix}\right],\quad 0<b<p,\ b\equiv a\pmod p,\ \alpha,\dots,\eta\in\Bbb Z_p.
\]
Assume that $P\in\text{GL}(M)$ such that
\begin{equation}\label{3.9}
PA(b,\alpha,\dots,\eta)=A(b,\alpha',\dots,\eta')P.
\end{equation}
Write
\[
P=\left[\begin{matrix} x&pY\cr Z&W\end{matrix}\right],
\]
where $x\in\Bbb Z_{p^2}^\times$, $Y\in\text{M}_{1\times 2}(\Bbb Z_p)$, $Z\in\text{M}_{2\times 1}(\Bbb Z_p)$, $W\in\text{GL}(2,\Bbb Z_p)$. Then \eqref{3.9} becomes
\[
\left[
\begin{matrix}
px\alpha+pY\left[\begin{matrix} \delta \cr \eta\end{matrix}\right]& px[\beta, \gamma] \vspace{2mm}\cr W\left[\begin{matrix} \delta\cr \eta\end{matrix}\right]&0\end{matrix}\right]=
\left[
\begin{matrix}
p\alpha'x+p[\beta',\gamma']z& p[\beta', \gamma']W \vspace{2mm} \cr \left[\begin{matrix} \delta'\cr \eta'\end{matrix}\right]x&0\end{matrix}\right].
\]
Over $\Bbb Z_p$, this becomes
\[
\left[
\begin{matrix}
x\alpha+Y\left[\begin{matrix} \delta\cr \eta\end{matrix}\right]& x[\beta, \gamma] \vspace{2mm}\cr W\left[\begin{matrix} \delta\cr \eta\end{matrix}\right]&0\end{matrix}\right]=
\left[
\begin{matrix}
\alpha'x+[\beta',\gamma']z& [\beta', \gamma']W \vspace{2mm}\cr \left[\begin{matrix} \delta'\cr \eta'\end{matrix}\right]x&0\end{matrix}\right].
\]
It is more convenient to write the above equation as 
\[
\left[\begin{matrix} x&Y\cr 0&W\end{matrix}\right]\left[\begin{matrix}[\beta,\gamma]&\alpha\cr 0&\left[\begin{matrix} \delta\cr \eta\end{matrix}\right]\end{matrix}\right]
=\left[\begin{matrix}[\beta',\gamma']&\alpha'\cr 0&\left[\begin{matrix} \delta'\cr \eta'\end{matrix}\right]\end{matrix}\right]\left[\begin{matrix} W&Z\cr 0&x\end{matrix}\right].
\]
So, $A(b,\alpha,\dots,\eta)$ and $A(b,\alpha',\dots,\eta')$ are conjugate if and only if 
\begin{equation}\label{3.10}
\left[\begin{matrix} x&Y\cr 0&W\end{matrix}\right]\left[\begin{matrix}[\beta,\gamma]&\alpha\cr 0&\left[\begin{matrix} \delta\cr \eta\end{matrix}\right]\end{matrix}\right]
\left[\begin{matrix} W^{-1}&Z\cr 0&x^{-1}\end{matrix}\right]=\left[\begin{matrix}[\beta',\gamma']&\alpha'\cr 0&\left[\begin{matrix} \delta'\cr \eta'\end{matrix}\right]\end{matrix}\right]
\end{equation}
for some $x\in\Bbb Z_p^\times$, $W\in\text{GL}(2,\Bbb Z_p)$, $Y\in\text{M}_{1\times 2}(\Bbb Z_p)$, $Z\in\text{M}_{2\times 1}(\Bbb Z_p)$.

Let 
\[
\mathcal M=\Bigl\{
\left[\begin{matrix}[\beta,\gamma]&\alpha\cr 0&\left[\begin{matrix} \delta\cr \eta\end{matrix}\right]\end{matrix}\right]:\alpha,\dots,\eta\in\Bbb Z_p\Bigr\}.
\]
For 
\[
A=\left[\begin{matrix}[\beta,\gamma]&\alpha\cr 0&\left[\begin{matrix} \delta\cr \eta\end{matrix}\right]\end{matrix}\right],\quad 
A'=\left[\begin{matrix}[\beta',\gamma']&\alpha'\cr 0&\left[\begin{matrix} \delta'\cr \eta'\end{matrix}\right]\end{matrix}\right]\in\mathcal M,
\]
say $A\sim A'$ if \eqref{3.10} is satisfied. It is easy to see that the $\sim$ equivalence classes in $\mathcal M$ are represented by
\begin{itemize}
  \item [(1)] $\left[\begin{matrix}[0\ 0]&\alpha\cr &\left[\begin{matrix} 0\cr 0\end{matrix}\right]\end{matrix}\right],\quad \alpha\in\Bbb Z_p$,
  \medskip
  \item [(2)] $\left[\begin{matrix}[0\ 0]\cr &\left[\begin{matrix} 1\cr 0\end{matrix}\right]\end{matrix}\right]$,
  \medskip
  \item [(3)] $\left[\begin{matrix}[1\ 0]\cr &\left[\begin{matrix} 0\cr \eta\end{matrix}\right]\end{matrix}\right],\quad \eta\in\Bbb Z_p$,
  \medskip
  \item [(4)] $\left[\begin{matrix}[1\ 0]\cr &\left[\begin{matrix} 1\cr 0\end{matrix}\right]\end{matrix}\right]$.
\end{itemize}  
These correspond to the matrices in (i.1) -- (i.4).

(ii) Elements of $\text{GL}(M)_f^{(1)(1,1)}$ are conjugate for matrices of the form 
\[
A(b,\alpha,\dots,\eta):=\!\! \left[\begin{matrix} b\cr &\kern-1mm  b&\kern-1mm  1\cr &&\kern-1mm  b\end{matrix}\right]+ \left[\begin{matrix} p\alpha&\kern-1mm p\beta& \kern-1mm p\gamma\cr \delta&\!0&\!0\cr \eta&\!0&\!0\end{matrix}\right],\ 0<b<p,\ b\equiv a
\ (\text{mod}\, p),\ \alpha,\dots,\eta\in\Bbb Z_p.
\]
Assume that $P\in\text{GL}(M)$ such that
\begin{equation}\label{3.11}
PA(b,\alpha,\dots,\eta)=A(b,\alpha',\dots,\eta')P.
\end{equation}
Write
\[
P=\left[\begin{matrix} x&pY\cr Z&W\end{matrix}\right],
\]
where $x\in\Bbb Z_{p^2}^\times$, $Y\in\text{M}_{1\times 2}(\Bbb Z_p)$, $Z\in\text{M}_{2\times 1}(\Bbb Z_p)$, $W\in\text{GL}(2,\Bbb Z_p)$.
Since $W$ commutes with $\left[\begin{smallmatrix} b&1\cr &b\end{smallmatrix}\right]$, we have $W=\left[\begin{smallmatrix} c&d\cr 0&c\end{smallmatrix}\right]$.
Equation~\eqref{3.11} is equivalent to
\[
\begin{split}
&\left[\begin{matrix}
0&pY\left[\begin{matrix} 0&1\cr &0\end{matrix}\right] \vspace{2mm}\cr
0&W\left[\begin{matrix} 0&1\cr &0\end{matrix}\right]
\end{matrix}\right]+\left[
\begin{matrix}
px\alpha+pY\left[\begin{matrix}\delta\cr \eta\end{matrix}\right]& px[\beta,\gamma]\vspace{2mm} \cr
W\left[\begin{matrix}\delta\cr \eta\end{matrix}\right]&0\end{matrix}\right]  \cr
=\;&\left[
\begin{matrix}
0&0 \vspace{2mm} \cr
\left[\begin{matrix} 0&1\cr &0\end{matrix}\right]Z& \left[\begin{matrix} 0&1\cr &0\end{matrix}\right]W\end{matrix}\right]+\left[
\begin{matrix}
p\alpha'x+p[\beta',\gamma']Z& p[\beta',\gamma']W \vspace{2mm} \cr
\left[\begin{matrix}\delta'\cr \eta'\end{matrix}\right]x&0\end{matrix}\right].
\end{split}
\]
Over $\Bbb Z_p$, this becomes
\[
\left[\begin{matrix}
0& \kern-1mm  Y\!\left[\begin{matrix} 0&\kern-1mm 1\cr &\kern-1mm  0\end{matrix}\right] \vspace{2mm} \cr
0& 0
\end{matrix}\right]+\left[
\begin{matrix}
x\alpha\!+\!Y\!\left[\begin{matrix}\delta\cr \eta\end{matrix}\right]& \kern-1mm  x[\beta,\gamma] \vspace{2mm} \cr
W\left[\begin{matrix}\delta\cr \eta\end{matrix}\right]&0\end{matrix}\right] 
\!=\!
\left[
\begin{matrix}
0& \kern-1mm  0 \vspace{2mm} \cr
\left[\begin{matrix} 0&\kern-1mm  1\cr &\kern-1mm  0\end{matrix}\right]\!Z& \kern-1mm  0\end{matrix}\right]+\left[
\begin{matrix}
\alpha'x\!+\![\beta',\gamma']Z&\kern-1mm  [\beta',\gamma']W \vspace{2mm} \cr
\left[\begin{matrix}\delta'\cr \eta'\end{matrix}\right]x&0\end{matrix}\right],
\]
which can be written as 
\[
\left[\begin{matrix} x&Y\cr 0&W\end{matrix}\right]\left[\begin{matrix}[\beta,\gamma]&\alpha\vspace{1mm} \cr 0&\left[\begin{matrix}\delta\cr \eta\end{matrix}\right]\end{matrix}\right]
+\left[\begin{matrix}Y\left[\begin{matrix} 0&1\cr &0\end{matrix}\right]&0 \vspace{2mm} \cr 0&-\left[\begin{matrix} 0&1\cr &0\end{matrix}\right]Z\end{matrix}\right]=
\left[\begin{matrix}[\beta',\gamma']&\alpha' \vspace{1mm} \cr 0&\left[\begin{matrix}\delta'\cr \eta'\end{matrix}\right]\end{matrix}\right]
\left[\begin{matrix} W&Z\cr 0&x\end{matrix}\right].
\]
So, $A(b,\alpha,\dots,\eta)$ and $A(b,\alpha',\dots,\eta')$ are conjugate if and only if 
\begin{equation}\label{3.12}
\left[\begin{matrix} x&\kern-1mm Y\cr 0&\kern-1mm W\end{matrix}\right]\!\!
\left[\begin{matrix}[\beta,\gamma]&\kern-1.5mm  \alpha \vspace{1mm} \cr 0&\kern-1.5mm  \left[\begin{matrix}\delta\cr \eta\end{matrix}\right]\end{matrix}\right]\!\!
\left[\begin{matrix} W^{-1}&\kern-1.5mm Z \vspace{1mm} \cr 0&\kern-1.5mm x^{-1}\end{matrix}\right]+\left[
\begin{matrix}
Y\!\left[\begin{matrix} 0&\kern-1mm 1\cr &\kern-1mm  0\end{matrix}\right]\! W^{-1}& Y\left[\begin{matrix} 0&1\cr &0\end{matrix}\right]Z \vspace{2mm} \cr
0&\left[\begin{matrix} 0&\kern-1mm 1\cr &\kern-1mm 0\end{matrix}\right]WZ\end{matrix}\right]\!=\!\left[\begin{matrix}[\beta',\gamma']&\kern-1.5mm \alpha' \vspace{1mm} \cr 0&\kern-1.5mm \left[\begin{matrix}\delta'\cr \eta'\end{matrix}\right]\end{matrix}\right]
\end{equation}
for some $x\in\Bbb Z_p^\times$, $Y\in\text{M}_{1\times 2}(\Bbb Z_p)$, $Z\in\text{M}_{2\times 1}(\Bbb Z_p)$, $W=\left[\begin{smallmatrix} c&d\cr &c\end{smallmatrix}\right]\in\text{GL}(2,\Bbb Z_p)$.

Let
\[
\mathcal M=\Bigl\{ \left[\begin{matrix}[\beta,\gamma]&\alpha\vspace{1mm} \cr 0&\left[\begin{matrix}\delta\cr \eta\end{matrix}\right]\end{matrix}\right]:\alpha,\dots,\eta\in\Bbb Z_p\Bigr\}.
\]
For 
\[
A=\left[\begin{matrix}[\beta,\gamma]&\alpha \vspace{1mm} \cr 0&\left[\begin{matrix}\delta\cr \eta\end{matrix}\right]\end{matrix}\right],\quad
A'=\left[\begin{matrix}[\beta',\gamma']&\alpha' \vspace{1mm} \cr 0&\left[\begin{matrix}\delta'\cr \eta'\end{matrix}\right]\end{matrix}\right]\in\mathcal M,
\]
say $A\sim A'$ if \eqref{3.12} is satisfied. It remains to determine the representatives of the $\sim$~equivalence classes in $\mathcal M$.

In \eqref{3.12}, we may assume $W=\left[\begin{smallmatrix} 1&d\cr &1\end{smallmatrix}\right]$ by replacing $x,Y,W,Z$ with  
$\frac 1c x, \frac 1c Y, \frac 1c W,cZ$, respectively. Let $Y=[y_1,y_2]$, $Z=\left[\begin{smallmatrix} z_1\cr z_2\end{smallmatrix}\right]$. Then \eqref{3.12} becomes
\begin{equation}\label{3.13}
\left[
\begin{matrix}
x\beta& -x\beta d\!+\!x\gamma\!+\!y_1& \alpha+x\beta z_1+x\gamma z_2\!+\!y_1\delta x^{-1}\!+\!y_2\eta x^{-1}\!+\!y_1z_2\cr
0&0&\delta x^{-1}+d\eta x^{-1}+z_2\cr
0&0& \eta x^{-1}\end{matrix}\right]\!=\!\left[
\begin{matrix}
\beta'&\kern-1mm \gamma'&\kern-1mm \alpha'\cr
0&\kern-1mm 0&\kern-1mm \delta'\cr
0&\kern-1mm 0&\kern-1mm \eta'\end{matrix}\right].
\end{equation}

We claim that the $\sim$ equivalence classes in $\mathcal M$ are represented by
\begin{itemize}
  \item [(1)] $\left[\begin{matrix}[0\ 0]&\alpha\cr &\left[\begin{matrix}0\cr 0\end{matrix}\right]\end{matrix}\right],\quad \alpha\in\Bbb Z_p$,
  \medskip
  \item [(2)] $\left[\begin{matrix}[0\ 0]&0\cr &\left[\begin{matrix}0\cr 1\end{matrix}\right]\end{matrix}\right]$,
  \medskip
  \item [(3)] $\left[\begin{matrix}[1\ 0]&0\cr &\left[\begin{matrix}0\cr \eta\end{matrix}\right]\end{matrix}\right],\quad \eta\in\Bbb Z_p$.
\end{itemize} 
The proof of (ii) will be complete when this claim is proved.

First, it is clear that matrices from different families of (1) -- (3) are not $\sim$~equivalent. So it remains to show that every $A\in\mathcal M$ can be brought into one of the ``canonical forms'' in (1) -- (3) through $\sim$ equivalence, and the entries in the canonical form are uniquely determined by $A$.

Let
\[
A=\left[\begin{matrix}[\beta,\gamma]&\alpha\cr &\left[\begin{matrix}\delta\cr \eta\end{matrix}\right]\end{matrix}\right]\in\mathcal M.
\]

First assume $\beta=0$ and $\eta=0$. Then
\[
\text{LHS of \eqref{3.13}}=\left[\begin{matrix}[0\ 0]&\alpha'\cr 0&\left[\begin{matrix}0\cr 0\end{matrix}\right]\end{matrix}\right]
\]
if and only if
\begin{equation}\label{3.14}
\begin{cases}
x\gamma+y_1=0,\cr
\delta x^{-1}+z_2=0,
\end{cases}
\end{equation}
and
\begin{equation}\label{3.15}
\alpha'=\alpha+\gamma\delta.
\end{equation}
System \eqref{3.14} has a solution $(x,y_1,z_2)=(1,-\gamma,-\delta)$. Equation~\eqref{3.15} shows that $\alpha'$ is uniquely determined by $A$.

Next, assume $\beta=0$ and $\eta\ne 0$. Then
\[
\text{LHS of \eqref{3.13}}=\left[\begin{matrix}[0\ 0]&0\cr &\left[\begin{matrix}0\cr 1\end{matrix}\right]\end{matrix}\right]
\]
if and only if 
\begin{equation}\label{3.16}
\begin{cases}
x\gamma+y_1=0,\cr
\alpha+x\gamma z_2+y_1\delta x^{-1}+y_2\eta x^{-1}+y_1 z_2=0,\cr
\delta x^{-1}+d\eta x^{-1}+z_2=0,\cr
\eta x^{-1}=1.
\end{cases}
\end{equation}
System \eqref{3.16} has a solution $(d,x,y_1,y_2,z_2)$. (Let $d$ be arbitrary. Solve for $x$ form the last equation, $y_1$ from the first equation, $z_2$ from the third equation, and $y_2$ from the second equation.) 

Finally, assume $\beta\ne 0$. Then  
\[
\text{LHS of \eqref{3.13}}=\left[\begin{matrix}[1\ 0]&0\cr &\left[\begin{matrix}0\cr \eta'\end{matrix}\right]\end{matrix}\right]
\]
if and only if 
\begin{equation}\label{3.17}
\begin{cases}
x\beta=1,\cr
-x\beta d+x\gamma+y_1=0,\cr
\alpha+x\beta z_1+x\gamma z_2+y_1\delta x^{-1}+y_2\eta x^{-1}+y_1z_2=0,\cr
\delta x^{-1}+d\eta x^{-1}+z_2=0,
\end{cases}
\end{equation}and
\begin{equation}\label{3.18}
\eta'=\eta\beta.
\end{equation}
System~\eqref{3.17} has a solution $(d,x,y_1,y_2,z_1,z_2)$. (Let $y_1$ and $y_2$ be arbitrary. Solve for $x$ from the first equation, $d$ from the second equation, $z_2$ from the last equation, and $z_1$ from the third equation.) Equation~\eqref{3.18} shows that $\eta'$ is uniquely determined by $A$.
\end{proof}

\subsection{Classification of $\Lambda$-modules of order $p^n$, $n\le 4$}\

The classification of $\Lambda$-modules of order $p^n$, $n\le 4$, is obtained by combining the results in 3.2 -- 3.6 and using Theorem~\ref{T2.1}. A complete description of the classification is given in Table~\ref{Tb1} in the appendix. From Table~\ref{Tb1} we find that the number of nonisomorphic $\Lambda$-modules of order $p^n$ is
\begin{equation}\label{3.19}
\begin{cases}
1&\text{if}\ n=0,\cr 
p-1&\text{if}\ n=1,\cr
2p^2-p-1&\text{if}\ n=2,\cr
3p^3-2p^2-1&\text{if}\ n=3,\cr
5p^4-2p^3-2p-1&\text{if}\ n=4.
\end{cases}
\end{equation}


\section{Finite Alexander Quandles}

By Theorem~\ref{T1.1}, the classification of finite Alexander quandles $M$ is the same as the classification of finite $\Lambda$-modules of the form $(1-t)M$.
First, the following question has to be answered: Given a finite $\Lambda$-module $N$ and an integer $l>0$, does there exist a $\Lambda$-module $M\supset N$ such that
$(1-t)M=N$ and $|M/N|=l$? Assume that such an $M$ exists. Since
\[
M/(1-t)M\overset{1-t}{\longrightarrow} (1-t)M/(1-t)^2M=N/(1-t)N
\]
is an onto $\Lambda$-map, we have $|N/(1-t)N|\bigm | l$. We will see in Theorem~\ref{T4.3} that $|N/(1-t)N|\bigm | l$ is also a sufficient condition for the existence of $M$.

\begin{lem}\label{L4.1}
Let $N\subset M$ be abelian groups and let $\alpha:N\to N$ and $\overline\alpha:M\to N$ be $\Bbb Z$-maps such that $\overline\alpha|_N=\alpha$. If $1-\alpha\in\text{\rm Aut}(N)$, then $1-\overline\alpha\in\text{\rm Aut}(M)$.
\end{lem}
 
\begin{proof}
We first show that $1-\overline\alpha$ is 1-1. Let $x\in\ker(1-\overline\alpha)$. Then
\[
0=\overline\alpha(0)=\overline\alpha(x-\overline\alpha(x))=\overline\alpha(x)-\alpha(\overline\alpha(x))=(1-\alpha)(\overline\alpha(x)).
\]
Since $1-\alpha\in\text{Aut}(N)$, we have $\overline\alpha(x)=0$. Thus $x=[(1-\overline\alpha)+\overline\alpha](x)=0$.

Now we show that $1-\overline\alpha:M\to M$ is onto. Let $y\in M$. Since $\overline\alpha(y)\in N$ and $1-\alpha\in\text{Aut}(N)$, there exists $x\in N$ such that $(1-\alpha)(x)=\overline\alpha(y)$. Then $y=y-\overline\alpha(y)+x-\alpha(x)=(1-\overline\alpha)(y+x)$.
\end{proof}

\begin{thm}\label{T4.2}
Let $N$ be a finite abelian group and $\alpha\in\text{\rm End}_{\Bbb Z}(N)$. Then there exist a finite abelian group $M\supset N$ with $|M/N|=|N/\alpha(N)|$ and an onto homomorphism $\overline\alpha:M\to N$ such that $\overline\alpha|_N=\alpha$.
\end{thm}

\begin{proof}
We may assume that $N$ is a finite abelian $p$-group. If $\alpha(N)=N$, there is nothing to prove. So assume $\alpha(N)\ne N$. Let $|N/\alpha(N)|=p^k$ and $\text{rank}\,N=r$ (the number of cyclic summands in a decomposition of $N$). We will inductively construct finite abelian $p$-groups $N=M_0\subset M_1\subset\cdots\subset M_k$ and $\Bbb Z$-maps
$\alpha_i:M_i\to N$, $0\le i\le k$, such that $\alpha_0=\alpha$, $|M_{i+1}/M_i|=|\alpha_{i+1}(M_{i+1})/\alpha_i(M_i)|=p$, $\text{rank}\,M_i=r$, and $\alpha_{i+1}|_{M_i}=\alpha_i$. Then $M=M_k$ and $\overline\alpha=\alpha_k$ have the desired property.

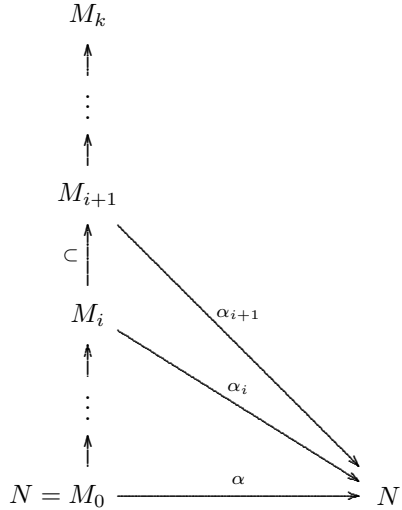
\[
\beginpicture
\setcoordinatesystem units <4mm,4mm> point at 0 0
\arrow <4pt> [0.3, 0.67] from 1 0 to 9 0
\arrow <4pt> [0.3, 0.67] from 1 5.5 to 9 0.5
\arrow <4pt> [0.3, 0.67] from 1 9 to 9 1
\arrow <4pt> [0.3, 0.67] from 0 1 to 0 2
\arrow <4pt> [0.3, 0.67] from 0 4 to 0 5
\arrow <4pt> [0.3, 0.67] from 0 7 to 0 9
\arrow <4pt> [0.3, 0.67] from 0 11 to 0 12
\arrow <4pt> [0.3, 0.67] from 0 14 to 0 15

\put {$ M_k$} at 0 16
\put {$ M_{i+1}$} at 0 10 
\put {$ M_i$} at 0 6
\put {$ N$} at 10 0
\put {$ N=M_0$} [r] at 0.6 0
\put {$\vdots$} at 0 3.2
\put {$\vdots$} at 0 13.2
\put {$\scriptstyle\subset$} [r] at -0.3 8
\put {$\scriptstyle\alpha$} [b] at 5 0.4
\put {$\scriptstyle\alpha_i$} [b] at 5 3.4
\put {$\scriptstyle\alpha_{i+1}$} [b] at 5 5.8
\endpicture
\] 

\bigskip

Let $0\le i<k$ and assume that $M_i$ and $\alpha_i$ have been constructed. We now construct $M_{i+1}$ and $\alpha_{i+1}$.

We claim that the mapping $M_i/pM_i\to N/pN$ induced by $\alpha_i$ is not 1-1. Otherwise, since $|M_i/pM_i|=p^r=|N/pN|$, the mapping is also onto. Then for each $x\in N$, there exist $y_0\in M_i$ and $x_0\in N$ such that 
\[
x=\alpha_i(y_0)+px_0.
\]
In the same way, $x_0=\alpha_i(y_1)+px_1$ for some $y_1\in M_i$ and $x_1\in N$. Continuing this way, we can write
\[
x=\alpha_i(y_0+py_1+\cdots+p^ny_n)+p^{n+1}x_n,\quad y_0,\dots,y_n\in M_i,\ x_n\in N.
\]
Choose $n$ large enough such that $p^{n+1}x_n=0$. Then $x=\alpha_i(y_0+py_1+\cdots+p^ny_n)\in\alpha_i(M_i)$. So $N=\alpha_i(M_i)$, which is a contradiction since
$|\alpha_i(M_i)/\alpha(N)|=p^i<p^k=|N/\alpha(N)|$.

By the above claim, there exists $a\in M_i\setminus pM_i$ such that $\alpha_i(a)\in pN$. Write $\alpha_i(a)=pb$ for some $b\in N$.

{\bf Case 1.} Assume $b\notin\alpha_i(M_i)$. 
Write $M_i=\Bbb Z_{p^{e_1}}\times\cdots\times\Bbb Z_{p^{e_r}}$, $e_1\ge\cdots\ge e_r>0$. 
Since $a\in M_i\setminus pM_i$, we may assume $a=(pw,1,0)$, where $w\in\Bbb Z_{p^{e_1}}\times\cdots\times\Bbb Z_{p^{e_{s-1}}}$ for some $0\le s<r$.
Let $M_{i+1}=A\times\Bbb Z_{p^{e_s+1}}\times B$ where $A=\Bbb Z_{p^{e_1}}\times\cdots\times\Bbb Z_{p^{e_{s-1}}}$, 
$B=\Bbb Z_{p^{e_{s+1}}}\times\cdots\times\Bbb Z_{p^{e_r}}$.
Define $\Bbb Z$-maps
\[
\begin{array}{cccc}
\iota:&M_i=A\times\Bbb Z_{p^{e_s}}\times B &\longrightarrow&M_{i+1}=A\times\Bbb Z_{p^{e_s+1}}\times B\cr
&(x,y,z)&\longmapsto&(x,py,z)
\end{array}
\]
\[
\begin{array}{cccc}
\alpha_{i+1}:&M_{i+1}=A\times\Bbb Z_{p^{e_s+1}}\times B&\longrightarrow& N\cr
&(x,y,z)&\longmapsto&\alpha_i(x,0,z)+y(b-\alpha_i(w,0,0))
\end{array}
\]
Then $\iota$ is 1-1, $\alpha_{i+1}\iota=\alpha_i$, and $|M_{i+1}/M_i|=p=|\alpha_{i+1}(M_{i+1})/\alpha_i(M_i)|$. 

{\bf Case 2.} Assume $b\in\alpha_i(M_i)$, say, $b=\alpha_i(c)$, $c\in M_i$. Then $\alpha_i(a)=p\alpha_i(c)$. Let $a_1=a-pc$. Then $a_1\in M_i\setminus pM_i$ and $\alpha_i(a_1)=0$. Choose $b'\in N\setminus\alpha_i(M_i)$ such that $pb'\in\alpha_i(M_i)$. Write $pb'=\alpha_i(d)$, $d\in M_i$. Choose $\epsilon=0$ or $1$ such that $a':=d+\epsilon a_1\notin pM_i$. Then $\alpha_i(a')=pb'$. Now we are in Case 1 with $a',b'$ in place of $a,b$, respectively.
\end{proof}  

\begin{thm}\label{T4.3}
Let $N$ be a finite $\Lambda$-module with $|N/(1-t)N|=l$. Then there exists a finite $\Lambda$-module $M\supset N$ with $|M/N|=l$ and $(1-t)M=N$.
\end{thm}

\begin{proof}
Let $\alpha=1-t\in\text{End}_\Bbb Z(N)$. By Theorem~\ref{T4.2}, there exist a finite abelian group $M\supset N$ and an onto $\Bbb Z$-map $\overline\alpha:M\to N$ such that $\overline\alpha|_N=\alpha$ and $|M/N|=|N/\alpha(N)|$. Since $1-\alpha\in\text{Aut}_\Bbb Z(N)$, by Lemma~\ref{L4.1}, $1-\overline\alpha\in\text{Aut}_\Bbb Z(M)$. Make $M$ into a
$\Lambda$-module by defining
\[
tx=(1-\overline\alpha)(x),\quad x\in M.
\]
Then ${}_\Lambda N$ is a submodule of ${}_\Lambda M$ and $(1-t)M=\overline\alpha(M)=N$.
\end{proof}

\begin{cor}\label{C4.4}
Let $p$ be a prime and $n$ a positive integer. Let $\mathcal M_{p^n}$ be a complete set of nonisomorphic $\Lambda$-modules $N$ such that $|N|=p^i$, $|(1-t)N|=p^j$, $2i-j\le n$. For each $N\in\mathcal M_{p^n}$, let $M_N\supset N$ be a $\Lambda$-module with $|M_N|=p^n$ and $(1-t)M_N=N$. (The existence of $M_N$ is $M_N$ is guaranteed by Theorem~\ref{T4.3}.) Then $\{(M_N,*):N\in\mathcal M_{p^n}\}$ is a complete set of nonisomorphic Alexander quandles of order $p^n$.
\end{cor}

\begin{proof}
Given a $\Lambda$-module $N$ with $|N|=p^i$ and $|(1-t)N|=p^j$, it follows from Theorem~\ref{T4.3} that $n\ge 2i-j$ is a necessary and sufficient condition on $n$ for which there exists a $\Lambda$-module $M\supset N$ with $|M|=p^n$ and $(1-t)M=N$. Now the conclusion in the corollary follows from Theorem~\ref{T1.1}.
\end{proof}

The $\Lambda$-modules in $\mathcal M_{p^n}$, $n\le 4$, are contained in Table~\ref{Tb1}; to save space, we will not enumerate these modules separately. From Table~\ref{Tb1}
we find that the number of nonisomorphic Alexander quandles of order $p^n$ ($n\le 4$) is
\begin{equation}\label{4.1}
\begin{cases}
1&\text{if}\ n=0,\cr 
p-1&\text{if}\ n=1,\cr
2p^2-2p-1&\text{if}\ n=2,\cr
3p^3-4p^2+p-3&\text{if}\ n=3,\cr
5p^4-6p^3+p^2-6p-1&\text{if}\ n=4.
\end{cases}
\end{equation}
The number of nonisomorphic connected Alexander quandles of order $p^n$ ($n\le 4$), also from Table~\ref{Tb1}, is
\begin{equation}\label{4.2}
\begin{cases}
1&\text{if}\ n=0,\cr 
p-2&\text{if}\ n=1,\cr
2p^2-3p-1&\text{if}\ n=2,\cr
3p^3-6p^2+p&\text{if}\ n=3,\cr
5p^4-9p^3+p^2-2p+1&\text{if}\ n=4.
\end{cases}
\end{equation}

\noindent{\bf Remark.} 
\begin{itemize}
  \item [(i)] \eqref{4.1} agrees with the numbers of nonisomorphic Alexander quandles of order $\le 15$ in \cite{Nel03}.
  \item [(ii)] \eqref{4.2} with $n=2$ agrees with the results of \cite{Gra04, Nel03}; \eqref{4.2} with $p^n=2^4$ agrees with the number in \cite{Mur-Nel09}.
  \item [(iii)] \cite{Mur-Nel09} stated that the number of nonisomorphic Alexander quandles of order $16$ is $24$. Our result (\eqref{4.1} with $p^n=2^4$) is $23$. It appears that the two Alexander quandles in \cite{Mur-Nel09} with $\text{Im}(\text{Id}-\phi)=\Bbb Z_4\oplus \Bbb Z_2$ (notation of \cite{Mur-Nel09}) are isomorphic. Professor W. E. Clark computed the numbers of nonisomorphic quandles of order $p^n<2^8$ using a computer program; his results (with $n\le 4$) agree with \eqref{4.1}. 
  \end{itemize}


\section*{Appendix: Table}


\begin{table}[h]
\caption{Nonisomorphic $\Lambda$-modules of order $p^n$, $n\le 4$}\label{Tb1}
\vspace{-5mm}
\[
\begin{tabular}{c|l|l|c|c} 
\multicolumn{5}{l}{$n=0$} \\ \hline
$(M,+)$ & matrix of $t$ & \multicolumn{1}{c|}{$|(1-t)M|$} & number &  total  \\ \hline
$0$ & $[0]$ & $1$ & $1$ & $1$  \\ \hline      
\end{tabular}
\]

\[
\begin{tabular}{c|l|l|c|c}

\multicolumn{5}{l}{$n=1$} \\ \hline
$(M,+)$ & matrix of $t$ & \multicolumn{1}{c|}{$|(1-t)M|$} & number &  total  \\ \hline
\multirow{2}*{$\Bbb Z_p$} & \multirow{2}*{$[b],\ b\in\Bbb Z_p^\times$} & $p^0$ if $b=1$ & $1$ & \multirow{2}*{$p-1$} \\ \cline{3-4} 
& & $p^1$ if $b\ne 1$ & $p-2$  \\ \hline      
\end{tabular}
\]
\end{table}


\addtocounter{table}{-1}

\begin{sidewaystable}[h]
\caption{Nonisomorphic $\Lambda$-modules of order $p^n$, $n\le 4$ (continued)}
\vspace{-5mm}
\[
\begin{tabular}{c|l|l|c|c}
\multicolumn{5}{l}{$n=2$} \\ \hline
$(M,+)$ & \multicolumn{1}{c|}{matrix of $t$} & \multicolumn{1}{c|}{$|(1-t)M|$} & number &  total  \\ \hline
\multirow{3}*{$\Bbb Z_{p^2}$} & \multirow{3}*{$[b],\ b\in\Bbb Z_{p^2}^\times$} &  $p^0$ if $b=1$ & $1$ & 
\multirow{3}*{$p^2-p$} \\
\cline{3-4}
& & $p^1$ if $b\ne 1,\ b\equiv 1\ (p)$ & $p-1$ &  \\ \cline{3-4}             
& & $p^2$ if $b\not\equiv 1\ (p)$ & $(p-2)p$ & \\ \hline
\multirow{3}*{$\Bbb Z_p^2$} & \multirow{3}*{$\left[\begin{matrix} b\cr & c\end{matrix}\right],\ 0<b\le c<p$} &
$p^0$ if $b=c=1$ & $1$ & \multirow{3}*{$\textstyle\binom p 2$} \\ \cline{3-4}
& & $p^1$ if $b=1<c$ & $p-2$ & \\ \cline{3-4}
& & $p^2$ if $b>1$ & $\textstyle\binom{p-1}2$ & \\ \hline
\multirow{2}*{$\Bbb Z_p^2$} &  \multirow{2}*{$\left[\begin{matrix} b&1\cr & b\end{matrix}\right],\ b\in\Bbb Z_p^\times$} &
$p^1$ if $b=1$ \phantom{$\Bigm |$} & $1$ &  \multirow{2}*{$p-1$} \\ \cline{3-4}
& & $p^2$ if $b\ne 1$ \phantom{$\Bigm |$} & $p-2$ & \\ \hline
$\Bbb Z_p^2$ & $\left[\begin{matrix} 0&1\cr -b_0& -b_1\end{matrix}\right],\ X^2+b_1X+b_0\in\Bbb Z_p[X]\ \text{irr}$ & $p^2$ \phantom{$\begin{matrix} 1\cr 1\cr 1\end{matrix}$} & $\frac 12(p^2-p)$ & $\frac 12(p^2-p)$ \\ \hline
\end{tabular}
\]

\vskip 1cm

\[
\begin{tabular}{c|l|l|c|c}
\multicolumn{5}{l}{$n=3$} \\ \hline
$(M,+)$ & \multicolumn{1}{c|}{matrix of $t$} & \multicolumn{1}{c|}{$|(1-t)M|$} & number &  total  \\ \hline
\multirow{4}*{$\Bbb Z_{p^3}$} & \multirow{4}*{$[b],\ b\in\Bbb Z_{p^3}^\times$} &  $p^0$ if $b=1$ & $1$ & \multirow{4}*{$p^3-p^2$} \\ \cline{3-4}
& & $p^1$ if $b\ne 1,\ b\equiv 1\ (p^2)$ & $p-1$ &  \\ \cline{3-4}             
& & $p^2$ if $b\not\equiv 1\ (p^2),\  b\equiv 1\ (p)$ & $p(p-1)$ &  \\ \cline{3-4}
& & $p^3$ if $b\not\equiv 1\ (p)$ & $p^2(p-2)$ & \\ \hline
\multirow{6}*{$\Bbb Z_{p^2}\times \Bbb Z_p$} & \multirow{6}*{$\left[\begin{matrix} b\cr &c\end{matrix}\right],\ b\in\Bbb Z_{p^2}^\times,\ c\in\Bbb Z_p^\times$} &
$p^0$ if $b=1,\ c=1$ & $1$ & \multirow{6}*{$p(p-1)^2$} \\ \cline{3-4}
& & $p^1$ if $b\ne 1$, $b\equiv 1\ (p),\ c=1$ & \multirow{2}*{$2p-3$} & \\
& & \phantom{$p^1$} or $b=1,\ c\ne 1$ & & \\ \cline{3-4}
& & $p^2$ if $b\not\equiv 1\ (p),\ c=1$ & \multirow{2}*{$(p-2)(2p-1)$} & \\
& & \phantom{$p^2$} or $b\ne 1,\ b\equiv 1\ (p),\ c\ne 1$ & & \\ \cline{3-4}
& & $p^3$ if $b\not\equiv 1\ (p),\ c\ne 1$ & $p(p-2)^2$ \\ \hline
\end{tabular}
\]
\end{sidewaystable}

\newpage


\addtocounter{table}{-1}

\begin{sidewaystable}
\caption{Nonisomorphic $\Lambda$-modules of order $p^n$, $n\le 4$ (continued)}
\vspace{-5mm}

\[
\begin{tabular}{c|l|l|c|c}
\multicolumn{5}{l}{$n=3$ (continued)} \\ \hline
$(M,+)$ & \multicolumn{1}{c|}{matrix of $t$} & \multicolumn{1}{c|}{$|(1-t)M|$} & number &  total  \\ \hline
\multirow{2}*{$\Bbb Z_{p^2}\times \Bbb Z_p$} & \multirow{2}*{$\left[\begin{matrix} b&0\cr 1&b\end{matrix}\right]$, $0<b<p$} 
& $p^1$ if $b=1$ \phantom{$\Bigm |$} & $1$ &
\multirow{2}*{$p-1$} \\ \cline{3-4}
& & $p^3$ if $b\ne 1$ \phantom{$\Bigm |$} & $p-2$ & \\ \hline
\multirow{3}*{$\Bbb Z_{p^2}\times \Bbb Z_p$} & \multirow{3}*{$\left[\begin{matrix} b&p\cr \gamma&b\end{matrix}\right]$, $0<b<p,\ \gamma\in\Bbb Z_p$} &
$p^1$ if $b=1,\ \gamma=0$ & $1$ & \multirow{3}*{$p(p-1)$} \\ \cline{3-4}
& & $p^2$ if $b=1,\ \gamma\ne0$ & $p-1$ & \\ \cline{3-4}
& & $p^3$ if $b\ne 1$ & $p(p-2)$ & \\ \hline
\multirow{4}*{$\Bbb Z_p^3$} & \multirow{4}*{$\left[\begin{matrix} b\cr &c\cr &&d\end{matrix}\right]$, $0<b\le c\le d<p$} & 
$p^0$ if $b=c=d=1$ & $1$ & \multirow{4}*{$\textstyle\binom {p+1}3$} \\ \cline{3-4}
& & $p^1$ if $b=c=1<d$ & $p-2$ \\ \cline{3-4}
& & $p^2$ if $b=1<c$ & $\textstyle\binom{p-1}2$ & \\ \cline{3-4}
& & $p^3$ if $b>1$ & $\textstyle\binom p3$ & \\ \hline
\multirow{4}*{$\Bbb Z_p^3$} & \multirow{4}*{$\left[\begin{matrix} b&1\cr &b\cr &&c\end{matrix}\right]$, $b,c\in\Bbb Z_p^\times$} &
$p^1$ if $b=c=1$ & $1$ & \multirow{4}*{$(p-1)^2$} \\ \cline{3-4}
& & $p^2$ if $b=1,\ c\ne 1$ & \multirow{2}*{$2(p-2)$} & \\
& & \phantom{$p^2$} or $b\ne 1,\ c=1$ & & \\ \cline{3-4}
& & $p^3$ if $b\ne 1\ c\ne 1$ & $(p-2)^2$ & \\ \hline
\multirow{2}*{$\Bbb Z_p^3$} & \multirow{2}*{$\left[\begin{matrix} b&1\cr &b&1\cr &&b\end{matrix}\right],\ b\in\Bbb Z_p^\times$} & 
$p^2$ if $b=1$ \phantom{$\displaystyle\int$} & $1$ & \multirow{2}*{$p-1$} \\ \cline{3-4}
& & $p^3$ if $b\ne 1$ \phantom{$\displaystyle\int$} & $p-2$ \\ \hline 
$\Bbb Z_p^3$ & $\left[\begin{matrix} 0&1&0\cr 0&0&1\cr -b_0&-b_1&-b_2\end{matrix}\right]$, $X^3+b_2X^2+b_1X+b_0\in\Bbb Z_p[X]\ \text{irr}$ &
$p^3$ \phantom{$\begin{matrix} 1\cr 1\cr 1\cr 1\end{matrix}$} & $\frac 13(p^3-p)$ & $\frac 13(p^3-p)$ \\ \hline
\multirow{2}*{$\Bbb Z_p^3$} & \multirow{2}*{$\left[\begin{matrix} 0&1\cr -b_0&-b_1\cr &&c\end{matrix}\right]$, $X^2+b_1X+b_0\in\Bbb Z_p[X]\ \text{irr},\ c\in\Bbb Z_p^\times$} & $p^2$ if $c=1$ \phantom{$\displaystyle\int$} & $\frac 12(p^2-p)$ & \multirow{2}*{$\frac 12p(p-1)^2$} \\ \cline{3-4}
& & $p^3$ if $c\ne 1$ \phantom{$\displaystyle\int$} & $\frac 12(p^2-p)(p-2)$ & \\ \hline
\end{tabular}
\]
\end{sidewaystable}

\newpage

\addtocounter{table}{-1}

\begin{sidewaystable}
\caption{Nonisomorphic $\Lambda$-modules of order $p^n$, $n\le 4$ (continued)}
\vspace{-5mm}
\[
\begin{tabular}{c|l|l|c|c}
\multicolumn{5}{l}{$n=4$} \\ \hline
$(M,+)$ & \multicolumn{1}{c|}{matrix of $t$} & \multicolumn{1}{c|}{$|(1-t)M|$} & number &  total  \\ \hline
\multirow{5}*{$\Bbb Z_{p^4}$} & \multirow{5}*{$[b],\ b\in\Bbb Z_{p^4}^\times$} & $p^0$ if $b=1$ & $1$ & \multirow{5}*{$p^3(p-1)$} \\ \cline{3-4}
& & $p^1$ if $b\ne 1$, $b\equiv 1\ (p^3)$ & $p-1$ & \\ \cline{3-4}
& & $p^2$ if $b\not\equiv 1\ (p^3)$, $b\equiv 1\ (p^2)$ & $p(p-1)$ & \\ \cline{3-4}
& & $p^3$ if $b\not\equiv 1\ (p^2)$, $b\equiv 1\ (p)$ & $p^2(p-1)$ & \\ \cline{3-4}
& & $p^4$ if $b\not\equiv 1\ (p)$ & $p^3(p-2)$ & \\ \hline
\multirow{8}*{$\Bbb Z_{p^3}\times \Bbb Z_p$} & \multirow{8}*{$\left[\begin{matrix} b\cr &c\end{matrix}\right]$, $b\in\Bbb Z_{p^3}^\times,\ c\in\Bbb Z_p^\times$} &
$p^0$ if $b=1,\ c=1$ & $1$ & \multirow{8}*{$p^2(p-1)^2$} \\ \cline{3-4}
& & $p^1$ if $b\ne 1,\ b\equiv 1\ (p^2),\ c=1$ & \multirow{2}*{$2p-3$} & \\
& & \phantom{$p^1$} or $b=1,\ c\ne 1$ & & \\ \cline{3-4}
& & $p^2$ if $b\not\equiv 1\ (p^2)$, $b\equiv 1\ (p)$, $c=1$ & \multirow{2}*{$2(p-1)^2$} & \\
& & \phantom{$p^2$} or $b\ne 1,\ b\equiv 1\ (p^2)$, $c\ne 1$ & & \\ \cline{3-4}
& &  $p^3$ if $b\not\equiv 1\ (p),\ c=1$ & \multirow{2}*{$p(p-2)(2p-1)$} & \\
& & \phantom{$p^2$} or $b\not\equiv 1\ (p^2)$, $b\equiv 1\ (p),\ c\ne 1$ & & \\ \cline{3-4}
& & $p^4$ if $b\not\equiv 1\ (p),\ c\ne 1$ & $p^2(p-2)^2$ & \\ \hline 
\multirow{3}*{$\Bbb Z_{p^3}\times \Bbb Z_p$} & \multirow{3}*{$\left[\begin{matrix} b&0\cr 1&b\end{matrix}\right]$, $0<b<p^2,\ b\not\equiv 0\ (p)$} &
$p^1$ if $b=1$ & $1$ & \multirow{3}*{$p(p-1)$} \\ \cline{3-4}
& & $p^2$ if $b\ne 1,\ b\equiv 1\ (p)$ & $p-1$ & \\ \cline{3-4}
& & $p^4$ if $b\not\equiv 1\ (p)$ & $p(p-2)$ & \\ \hline
\multirow{4}*{$\Bbb Z_{p^3}\times \Bbb Z_p$} & \multirow{4}*{$\left[\begin{matrix} b&\kern-2mm p^2\cr \gamma& \kern -2mm b\end{matrix}\right]$, $0<b<p^2,\ b\not\equiv 0\ (p)$, $\gamma\in\Bbb Z_p$} & $p^1$ if $b=1,\ \gamma=0$ & $1$ & \multirow{4}*{$p^2(p-1)$} \\ \cline{3-4}
& & $p^2$ if $b\ne 1,\ b\equiv 1\ (p)$ & \multirow{2}*{$p^2-1$} & \\
& & \phantom{$p^2$} or $b=1,\ \gamma\ne 0$  & & \\ \cline{3-4}
& & $p^4$ if $b\not\equiv 1\ (p)$ & $p^2(p-2)$ \\ \hline
\multirow{6}*{$\Bbb Z_{p^2}^2$} & \multirow{6}*{$\left[\begin{matrix} b\cr &c\end{matrix}\right]$, $0<b\le c<p^2$} &
$p^0$ if $b=c=1$ & $1$ & \multirow{6}*{$\textstyle\binom{p(p-1)+1}2$} \\ \cline{3-4}
& & $p^1$ if $b=1<c$, $c\equiv 1\ (p)$ & $p-1$ & \\ \cline{3-4}
& & $p^2$ if $b=1,\ c\not\equiv 1\ (p)$ & \multirow{2}*{$p(p-2)+\textstyle\binom p2$} & \\
& & \phantom{$p^2$} or $b, c\ne 1$, $b,c\equiv 1\ (p)$ & & \\ \cline{3-4}
& & $p^3$ if $\{b,c\}\!=\!\{b_1,c_1\}$, $b_1\not\equiv 1\ (p)$, $c_1\ne 1,\ c_1\!\equiv \!1\, (p)$ & $p(p-1)(p-2)$ & \\ \cline{3-4}
& & $p^4$ if $b,c\not\equiv 1\ (p)$ & $\textstyle\binom{p(p-2)+1}2$ & \\ \hline
\multirow{3}*{$\Bbb Z_{p^2}^2$} & \multirow{3}*{$\left[\begin{matrix} b&p\cr 0&b\end{matrix}\right],\ b\in\Bbb Z_{p^2}^\times$} &
$p^1$ if $b=1$ & $1$ & \multirow{3}*{$p(p-1)$} \\ \cline{3-4}
& & $p^2$ if $b\ne 1,\ b\equiv 1\ (p)$ & $p-1$ & \\ \cline{3-4}\
& & $p^4$ if $b\not\equiv 1\ (p)$ & $p(p-2)$ & \\ \hline
\end{tabular}
\]
\end{sidewaystable}

\newpage 

\addtocounter{table}{-1}

\begin{sidewaystable}
\caption{Nonisomorphic $\Lambda$-modules of order $p^n$, $n\le 4$ (continued)}
\vspace{-5mm}
\[
\begin{tabular}{c|l|l|c|c}
\multicolumn{5}{l}{$n=4$ (continued)} \\ \hline
$(M,+)$ & \multicolumn{1}{c|}{matrix of $t$} & \multicolumn{1}{c|}{$|(1-t)M|$} & number &  total  \\ \hline
\multirow{2}*{$\Bbb Z_{p^2}^2$} & \multirow{2}*{$\left[\begin{matrix} b&p\cr -pb_0&b-pb_1\end{matrix}\right]$, $\begin{array}{l} 0<b<p,\cr X^2+b_1X+b_0\in\Bbb Z_p[X]\ \text{irr}\end{array}$}  &
$p^2$ if $b=1$ \phantom{$\Bigm |$} & $\frac 12(p^2-p)$ & \multirow{2}*{$\frac 12p(p-1)^2$}\\ \cline{3-4}
& & $p^4$ if $b\ne 1$ \phantom{$\Bigm |$} & $\frac 12(p^2-p)(p-2)$ & \\ \hline
\multirow{3}*{$\Bbb Z_{p^2}^2$} & \multirow{3}*{$\left[\begin{matrix} b+p\alpha&1\cr p\gamma &b\end{matrix}\right]$, $0<b<p$, $\alpha,\gamma\in\Bbb Z_p$} &
$p^2$ if $b=1,\ \gamma=0$ & $p$ &  \multirow{3}*{$p^2(p-1)$}\\ \cline{3-4}
& & $p^3$ if $b=1,\ \gamma\ne 0$ & $p(p-1)$ & \\ \cline{3-4}
& & $p^4$ if $b\ne 1$ & $p^2(p-2)$ & \\ \hline
$\Bbb Z_{p^2}^2$ & $\left[\begin{matrix} p\alpha&1+p\beta\cr -b_0&-b_1\end{matrix}\right]$, $\begin{array}{l} \alpha,\beta\in\Bbb Z_p,\ 0\le b_0,b_1<p \cr X^2+b_1X+b_0\in\Bbb Z_p[X]\ \text{irr}\end{array}$ & $p^4$ \phantom{$\begin{matrix} 1\cr 1\cr 1\end{matrix}$} & $\frac 12p^2(p^2-p)$ & $\frac 12p^2(p^2-p)$ \\ \hline
\multirow{9}*{$\Bbb Z_{p^2}\times\Bbb Z_p^2$} & \multirow{9}*{$\left[\begin{matrix} b\cr &c\cr &&d\end{matrix}\right]$, $b\in\Bbb Z_{p^2}^\times,\ 0<c\le d<p$} &  
$p^0$ if $b=1,\ c=d=1$ & $1$ & \multirow{9}*{$\frac 12p^2(p-1)^2$} \\ \cline{3-4}
& & $p^1$ if $b\ne 1,\ b\equiv 1\ (p)$, $c=d=1$ & \multirow{2}*{$2p-3$} & \\ 
& & \phantom{$p^2$} or $b=1,\ c=1<d$ & & \\ \cline{3-4} 
& & $p^2$ if $b\not\equiv 1\ (p)$, $c=d=1$ & \multirow{3}*{$\frac 12(p-2)(5p-3)$} & \\
& & \phantom{$p^2$} or $b\ne 1,\ b\equiv 1\ (p)$, $c=1<d$ & & \\
& & \phantom{$p^2$} or $b=1,\ c>1$ & & \\ \cline{3-4}
& & $p^3$ if $b\not\equiv 1\ (p)$, $c=1<d$ & \multirow{2}*{$\frac 12(p-2)(3p^2-6p+1)$} & \\
& & \phantom{$p^3$} or $b\ne 1,\ b\equiv 1\ (p)$, $c>1$ & &  \\ \cline{3-4}
& & $p^4$ if $b\not\equiv 1\ (p)$, $c>1$ & {$\frac 12p(p-1)(p-2)^2$} & \\ \hline
\multirow{4}*{$\Bbb Z_{p^2}\times\Bbb Z_p^2$} &  \multirow{4}*{$\left[\begin{matrix} b&0&0\cr 1&b&0\cr 0&0&c\end{matrix}\right]$, $0<b<p,\ c\in\Bbb Z_p^\times$} &
$p^1$ if $b=1,\ c=1$ & $1$ & \multirow{4}*{$(p-1)^2$} \\ \cline{3-4}
& & $p^2$ if $b=1,\ c\ne 1$ & $p-2$ & \\ \cline{3-4}
& & $p^3$ if $b\ne 1,\ c=1$ & $p-2$ & \\ \cline{3-4}
& & $p^4$ if $b\ne 1,\ c\ne 1$ & $(p-2)^2$ \\ \hline
\multirow{3}*{$\Bbb Z_{p^2}\times\Bbb Z_p^2$} &  \multirow{3}*{$\left[\begin{matrix} b&p&0\cr 0&b&0\cr \eta&0&b\end{matrix}\right]$, $0<b<p,\ \eta\in\Bbb Z_p$} &
$p^1$ if $b=1,\ \eta=0$ \phantom{$\Bigm |$} & $1$ & \multirow{3}*{$p(p-1)$} \\ \cline{3-4}
& & $p^2$ if $b=1,\ \eta\ne 0$ \phantom{$\Bigm |$} & $p-1$ & \\ \cline{3-4}
& & $p^4$ if $b\ne 1$ \phantom{$\Bigm |$} & $p(p-2)$ & \\ \hline
\multirow{2}*{$\Bbb Z_{p^2}\times\Bbb Z_p^2$} &  \multirow{2}*{$\left[\begin{matrix} b&p&0\cr 1&b&0\cr 0&0&b\end{matrix}\right],\ 0<b<p$} & 
$p^2$ if $b=1$ \phantom{$\displaystyle\int$} & $1$ &  \multirow{2}*{$p-1$} \\ \cline{3-4}
& & $p^4$ if $b\ne 1$ \phantom{$\displaystyle\int$} & $p-2$ & \\ \hline
\end{tabular}
\]
\end{sidewaystable}

\newpage

\addtocounter{table}{-1}

\begin{sidewaystable}
\caption{Nonisomorphic $\Lambda$-modules of order $p^n$, $n\le 4$ (continued)}
\vspace{-5mm}
\[
\begin{tabular}{c|l|l|c|c}
\multicolumn{5}{l}{$n=4$ (continued)} \\ \hline
$(M,+)$ & \multicolumn{1}{c|}{matrix of $t$} & \multicolumn{1}{c|}{$|(1-t)M|$} & number &  total  \\ \hline
\multirow{6}*{$\Bbb Z_{p^2}\times\Bbb Z_p^2$} &  \multirow{6}*{$\left[\begin{matrix} b\cr &c&1\cr &&c\end{matrix}\right]$, $b\in\Bbb Z_{p^2}^\times,\ c\in\Bbb Z_p^\times$} &$p^1$ if $b=1,\ c=1$ & $1$  &  \multirow{6}*{$p(p-1)^2$} \\ \cline{3-4}
& & $p^2$ if $b\ne 1,\ b\equiv 1\ (p)$, $c=1$ & \multirow{2}*{$2p-3$} & \\
& & \phantom{$p^2$} or $b=1,\ c\ne 1$ & & \\ \cline{3-4}
& & $p^3$ if $b\not\equiv 1\ (p),\ c=1$  & \multirow{2}*{$(p-2)(2p-1)$} & \\
& & \phantom{$p^3$} or $b\ne 1,\ b\equiv 1\ (p)$, $c\ne 1$ & & \\ \cline{3-4}
& & $p^4$ if $b\not\equiv 1\ (p),\ c\ne 1$ & $p(p-2)^2$ & \\ \hline
\multirow{2}*{$\Bbb Z_{p^2}\times\Bbb Z_p^2$} &  \multirow{2}*{$\left[\begin{matrix} b&0&0\cr 0&b&1\cr 1&0&b \end{matrix}\right],\ 0<b<p$} &
$p^2$ if $b=1$ \phantom{$\displaystyle\int$} & $1$ & \multirow{2}*{$p-1$} \\ \cline{3-4}
& & $p^4$ if $b\ne 1$ \phantom{$\displaystyle\int$} & $p-2$ & \\ \hline
\multirow{3}*{$\Bbb Z_{p^2}\times\Bbb Z_p^2$} &  \multirow{3}*{$\left[\begin{matrix} b&p&0\cr 0&b&1\cr \eta&0&b \end{matrix}\right]$, $0<b<p,\ \eta\in\Bbb Z_p$} &
  $p^2$ if $b=1,\ \eta=0$ \phantom{$\Bigm |$} & $1$ &  \multirow{3}*{$p(p-1)$} \\ \cline{3-4}
& & $p^3$ if $b=1,\ \eta\ne 0$ \phantom{$\Bigm |$} & $p-1$ & \\ \cline{3-4}
& & $p^4$ if $b\ne 1$ \phantom{$\Bigm |$} & $p(p-2)$ &  \\ \hline
\multirow{5}*{$\Bbb Z_p^4$} &  \multirow{5}*{$\left[\begin{matrix} b\cr &c\cr &&d\cr &&& e\end{matrix}\right]$, $0<b\le c\le d\le e<p$}  &
$p^0$ if $b=c=d=e=1$ & $1$ & \multirow{5}*{$\textstyle\binom{p+2}4$} \\ \cline{3-4}
& & $p^1$ if $b=c=d=1<e$ & $p-2$ &  \\ \cline{3-4}
& & $p^2$ if $b=c=1<d$ & $\textstyle\binom{p-1}2$ &  \\ \cline{3-4}
& & $p^3$ if $b=1<c$ & $\textstyle\binom p3$ &  \\ \cline{3-4}
& & $p^4$ if $b>1$ & $\textstyle\binom {p+1}4$ & \\ \hline
\multirow{6}*{$\Bbb Z_p^4$} & \multirow{6}*{$\left[\begin{matrix} b&1\cr &b\cr &&c\cr &&& d\end{matrix}\right]$, $b\in\Bbb Z_p^\times,\ 0<c\le d<p$} &
 $p^1$ if $b=1,\ c=d=1$ & $1$ & \multirow{6}*{$\frac 12p(p-1)^2$} \\ \cline{3-4}
& & $p^2$ if $b\ne 1,\ c=d=1$ & \multirow{2}*{$2(p-2)$} & \\
& & \phantom{$p^2$} or $b=1,\ c=1<d$ & & \\ \cline{3-4}
& & $p^3$ if $b\ne 1,\ c=1<d$ & \multirow{2}*{$\frac 12(p-2)(3p-5)$} & \\ 
& & \phantom{$p^3$} or $b=1,\ c>1$ & & \\ \cline{3-4}
& & $p^4$ if $b\ne 1,\ c>1$ & $\frac 12(p-1)(p-2)^2$  \\ \hline
\multirow{3}*{$\Bbb Z_p^4$} &   \multirow{3}*{$\left[\begin{matrix} b&1\cr &b\cr &&c&1\cr &&& c\end{matrix}\right]$, $0<b\le c<p$} &
$p^2$ if $b=c=1$ \phantom{$\Bigm |$} & $1$ &  \multirow{3}*{$\textstyle\binom p2$} \\ \cline{3-4}
& & $p^3$ if $b=1<c$ \phantom{$\Bigm |$} & $p-2$ & \\ \cline{3-4}
& & $p^4$ if $b>1$ \phantom{$\Bigm |$} & $\textstyle\binom{p-1}2$ &  \\ \hline
\end{tabular}
\]
\end{sidewaystable}

\newpage

\addtocounter{table}{-1}

\begin{sidewaystable}
\caption{Nonisomorphic $\Lambda$-modules of order $p^n$, $n\le 4$ (continued)}
\vspace{-5mm}
\[
\begin{tabular}{c|l|l|c|c}
\multicolumn{5}{l}{$n=4$ (continued)} \\ \hline
$(M,+)$ & \multicolumn{1}{c|}{matrix of $t$} & \multicolumn{1}{c|}{$|(1-t)M|$} & number &  total  \\ \hline
\multirow{4}*{$\Bbb Z_p^4$} &  \multirow{4}*{$\left[\begin{matrix} b&1\cr &b&1\cr &&b\cr &&& c\end{matrix}\right]$, $b,c\in\Bbb Z_p^\times$} &
$p^2$ if $b=c=1$ \phantom{$\Bigm |$} & $1$ & \multirow{4}*{$(p-1)^2$} \\ \cline{3-4}
& & $p^3$ if $b=1,\ c\ne 1$ & \multirow{2}*{$2(p-2)$} & \\
& & \phantom{$p^3$} or $b\ne 1,\ c=1$ & & \\ \cline{3-4}
& & $p^4$ if $b\ne 1,\ c\ne 1$ \phantom{$\Bigm |$} & $(p-2)^2$ & \\ \hline
\multirow{4}*{$\Bbb Z_p^4$} &  \multirow{2}*{$\left[\begin{matrix} b&1\cr &b&1\cr &&b&1\cr &&& b\end{matrix}\right],\ b\in\Bbb Z_p^\times$}  &
$p^3$ if $b=1$ \phantom{$\begin{matrix} \Bigm |\cr \bigm |\end{matrix}$} & $1$ & \multirow{4}*{$p-1$} \\ \cline{3-4}
& &  $p^4$ if $b\ne 1$ \phantom{$\begin{matrix} \bigm |\cr \Bigm |\end{matrix}$} & $p-2$ & \\ \hline
$\Bbb Z_p^4$ & $\left[\begin{matrix} 0&1\cr -b_0&-b_1\cr &&0&1\cr &&-c_0&-c_1\end{matrix}\right]$, $\begin{array}{l} X^2+b_1X+b_0,\ X^2+c_1X+c_0\in\Bbb Z_p[X]\ \text{irr}\cr (b_0,b_1)\le(c_0,c_1)\ \text{(*)} \end{array}$ & $p^4$ \phantom{$\begin{matrix} \Bigm |\cr \Bigm |\cr \Bigm |\end{matrix}$} & $\textstyle\binom{\frac12(p^2-p)+1}2$ & $\textstyle\binom{\frac12(p^2-p)+1}2$ \\ \hline
$\Bbb Z_p^4$ & $\left[\begin{matrix} 0&1&1&0\cr -b_0&-b_1&0&1\cr &&0&1\cr &&-b_0&-b_1\end{matrix}\right]$, $X^2+b_1X+b_0\in\Bbb Z_p[X]\ \text{irr}$ & $p^4$ &
\phantom{$\begin{matrix} \Bigm |\cr \Bigm |\cr \Bigm |\end{matrix}$} $\frac12(p^2-p)$ & $\frac12(p^2-p)$ \\ \hline
$\Bbb Z_p^4$ & $\left[\begin{matrix} 0&1&0&0\cr 0&0&1&0 \cr 0&0&0&1\cr -b_0&-b_1&-b_2&-b_3\end{matrix}\right]$, $X^4+b_3X^3+b_2X^2+b_1X+b_0\in\Bbb Z_p[X]\ \text{irr}$ &
$p^4$ \phantom{$\begin{matrix} \Bigm |\cr \Bigm |\cr \Bigm |\end{matrix}$} & $\frac14(p^4-p^2)$ & $\frac14(p^4-p^2)$ \\ \hline

\multicolumn{5}{l}{(*) $\le$ is a total order in $\Bbb Z_p^2$} \phantom{$\Bigm |$} \\
\end{tabular}
\]
\end{sidewaystable}

\newpage

\addtocounter{table}{-1}

\begin{sidewaystable}
\caption{Nonisomorphic $\Lambda$-modules of order $p^n$, $n\le 4$ (continued)}
\vspace{-5mm}
\[
\begin{tabular}{c|l|l|c|c}
\multicolumn{5}{l}{$n=4$ (continued)} \\ \hline
$(M,+)$ & \multicolumn{1}{c|}{matrix of $t$} & \multicolumn{1}{c|}{$|(1-t)M|$} & number &  total  \\ \hline
\multirow{4}*{$\Bbb Z_{p^2}\times\Bbb Z_p^2$} &  \multirow{4}*{$\left[\begin{matrix} b&p\cr \gamma&b \cr &&c \end{matrix}\right]$, $0<b<p,\ \gamma\in\Bbb Z_p$, $c\in\Bbb Z_p^\times,\ b\not\equiv c\ (p)$} & $p^2$ if $b\!=\!1,\ \gamma\!=\!0,\ c\!\ne\! 1$ & $p-2$ & \multirow{4}*{$p(p-1)(p-2)$} \\ \cline{3-4}
& & $p^3$ if $b\ne 1,\ c=1$ & \multirow{2}*{$(p-2)(2p-1)$} & \\
& & \phantom{$p^3$} or $b\!=\!1,\ \gamma\!\ne\! 0,\ c\!\ne\! 1$ & & \\ \cline{3-4}
& & $p^4$ if $b\ne1,\ c\ne 1$ & $p(p-2)(p-3)$ & \\ \hline
\multirow{3}*{$\Bbb Z_p^4$} & \multirow{2}*{$\left[\begin{matrix} 0&\kern -1.5mm 1 & \kern -1.5mm 0\cr 0& \kern -1.5mm 0& \kern -1.5mm 1\cr -b_0&\kern -2mm -b_1&\kern -2mm -b_2\cr &&&\kern -1.5mm c\end{matrix}\right]$, $\kern -1mm \begin{array}{l}X^3\!+\!b_2X^2\!+\!b_1X\!+\!b_0\in\Bbb Z_p[X]\ \text{irr},\cr c\in\Bbb Z_p^\times\end{array}$}  & $p^3$ if $c=1$  \phantom{$\begin{matrix} \bigm | \cr \Bigm | \end{matrix}$} & $\frac 13(p^3-p)$  & \multirow{3}*{$\frac 13(p^3-p)(p-1)$} \\ \cline{3-4}
& & $p^4$ if $c\ne 1$ \phantom{$\begin{matrix} \bigm | \cr \Bigm | \end{matrix}$}  & $\frac 13(p^3-p)(p-2)$ & \\ \hline
\multirow{3}*{$\Bbb Z_{p^2}\times\Bbb Z_p^2$} &  \multirow{3}*{$\left[\begin{matrix} b\cr &0&1\cr &-c_0&-c_1\end{matrix}\right],\ b\in\Bbb Z_{p^2}^\times$, $X^2+c_1X+c_0\in\Bbb Z_p[X]\ \text{irr}$} & $p^2$ if $b=1$ \phantom{$\Bigm |$} & $\frac 12(p^2-p)$ & \multirow{3}*{$\frac 12p^2(p-1)^2$} \\ \cline{3-4}
& & $p^3$ if $b\ne 1,\ b\equiv 1\ (p)$ \phantom{$\Bigm |$} & $\frac 12(p^2-p)(p-1)$ & \\ \cline{3-4} 
& & $p^4$ if $b\not\equiv 1\ (p)$ \phantom{$\Bigm |$} & $\frac 12(p^2-p)p(p-2)$ & \\ \hline
\multirow{3}*{$\Bbb Z_p^4$} &  \multirow{3}*{$\left[\begin{matrix} 0&\kern-1.5mm 1\cr -b_0&\kern-2mm -b_1\cr && c\cr &&&\kern -1.5mm d\end{matrix}\right]$, $\begin{array}{l}X^2+b_1X+b_0\in\Bbb Z_p[X]\ \text{irr},\cr 0<c\le d<p\end{array}$} & $p^2$ if $c=d=1$ \phantom{$\Bigm |$} & $\frac12(p^2-p)$ & \multirow{3}*{$\frac14p^2(p-1)^2$} \\ \cline{3-4}
& & $p^3$ if $c=1<d$ \phantom{$\Bigm |$} & $\frac12(p^2-p)(p-2)$ & \\ \cline{3-4}
& & $p^4$ if $c>1$ \phantom{$\Bigm |$} & $\frac12(p^2-p)\textstyle\binom{p-1}2$ & \\ \hline
\multirow{3}*{$\Bbb Z_p^4$} &  \multirow{2}*{$\left[\begin{matrix} b& \kern -1mm 1\cr &\kern -1.5mm b\cr &&\kern -1mm 0&\kern -1mm 1\cr &&\kern -1.5mm -c_0&\kern -1.5mm -c_1\end{matrix}\right]$, $b\in\Bbb Z_p^\times,\
 X^2+c_1X+c_0\in\Bbb Z_p[X]\ \text{irr}$} &  $p^3$ if $b=1$ \phantom{$\begin{matrix} \bigm | \cr \Bigm | \end{matrix}$} & $\frac12(p^2-p)$ & \multirow{3}*{$\frac12p(p-1)^2$} \\ \cline{3-4}
& & $p^4$ if $b\ne 1$ \phantom{$\begin{matrix} \bigm | \cr \Bigm | \end{matrix}$} & $\frac12(p^2-p)(p-2)$ & \\ \hline
\end{tabular}
\]
\end{sidewaystable}


\section*{Acknowledgment}

I thank Professor W. Edwin Clark for sharing his computational results on the numbers of nonisomorphic Alexander quandles.


\end{document}